\documentclass[11pt,reqno]{amsart}
\usepackage{pdflscape,afterpage}
\usepackage[depth=3]{bookmark}
\usepackage[dvipsnames]{xcolor}
\usepackage{amssymb, mathtools, booktabs}
\usepackage{hyperref, enumitem, nicematrix}
\usepackage{graphicx, subcaption, tikz, thmtools, wrapfig}
\usepackage[margin=15pt,labelfont={bf}]{caption}
\usepackage[capitalize,nameinlink,noabbrev]{cleveref}
\crefname{equation}{}{}
\Crefname{equation}{Equation}{Equations}
\usepackage{crossreftools}
\usetikzlibrary{arrows,backgrounds,calc,shapes.geometric,positioning,petri}
\usepackage[T1]{fontenc}
\usepackage[margin = .95in]{geometry}
\usepackage[mathscr]{euscript}
\hypersetup{
colorlinks,
final,
pdfauthor={Antoine Abram and Jose Bastidas},
pdftitle={$h^*$-polynomials of type C hypersimplices},
pdfkeywords={Alcoved polytopes, Big ascents, Coxeter groups, Weak order, Ehrhart h*-polynomials, hypersimplices},
urlcolor = cyan,
linkcolor = OliveGreen,
citecolor = Violet,
}
\usepackage[square,sort,comma,numbers]{natbib}
\usepackage{doi}
\newtheorem{theorem}{Theorem}[]
\newtheorem{lemma}[theorem]{Lemma}
\newtheorem{proposition}[theorem]{Proposition}
\newtheorem{corollary}[theorem]{Corollary}
\newtheorem{conjecture}[theorem]{Conjecture}
\theoremstyle{definition}
\newtheorem{definition}[theorem]{Definition}
\newtheorem{example}[theorem]{Example}
\newtheorem{remark}[theorem]{Remark}
\newcommand{\ol}[1]{\overline{#1}}
\newcommand{\wt}[1]{\widetilde{#1}}
\newcommand{\wh}[1]{\widehat{#1}}
\newcommand{\dessymb}{\wh{\!}}
\newcommand{\ds}[1]{\underset{\dessymb}{#1}}
\newcommand{\cdessymb}{\!^{\;\circ}}
\newcommand{\RR}{\mathbb{R}}
\newcommand{\ZZ}{\mathbb{Z}}

\newcommand{\HH}{\mathcal{H}}

\newcommand{\defn}[1]{{\color{RoyalBlue}\bf#1}}
\newcommand{\defnn}[1]{{\color{RoyalBlue}#1}}
\newcommand{\red}[1]{{\color{red}#1}}

\newcommand{\hypoct}{\mathfrak{B}}
\newcommand{\highlight}[1]{\colorbox{RoyalBlue!20}{$\displaystyle#1$}}
\newcommand\YS{\begin{tikzpicture}[scale = .1]}
\newcommand\YSF{\begin{tikzpicture}[scale = .4]}
\newcommand\YB[2]{\draw[fill = yellow] (#1,#2) rectangle (#1+1,#2+1);}
\newcommand\YBF[3]{\draw[fill = yellow] (#1,#2) rectangle (#1+1,#2+1) node[pos=.5] {\scriptsize $#3$};}
\newcommand\YF{\end{tikzpicture}}
\DeclareMathOperator{\Conv}{Conv}
\DeclareMathOperator{\Vol}{Vol}
\DeclareMathOperator{\BAsc}{BAsc}
\DeclareMathOperator{\basc}{basc}

\DeclareMathOperator{\CDes}{CDes}
\DeclareMathOperator{\cdes}{cdes}

\DeclareMathOperator{\des}{des}

\DeclareMathOperator{\exc}{exc}
\DeclareMathOperator{\fdes}{fdes}
\DeclareMathOperator{\fexc}{fexc}

\DeclareMathOperator{\fneg}{neg}
\DeclarePairedDelimiter{\floor}{\lfloor}{\rfloor}

\DeclarePairedDelimiter{\angbra}{\langle}{\rangle}
\DeclarePairedDelimiter{\dbra}{[\![}{]\!]}

\title[$h^*$-polynomials of type C hypersimplices]{The $h^*$-polynomials of type C hypersimplices}

\address{
Laboratoire d’Algèbre, de Combinatoire et d’Informatique Mathématique (LACIM)\\
Universit\'e du Qu\'ebec \`a Montr\'eal\\
CP 8888 Succ. Centre-Ville\\
Montr\'eal, Qu\'ebec, H3C 3P8\\ Canada}

\author[A.~Abram]{Antoine~Abram}
\address[Antoine Abram]{}
\email{abram.antoine@courrier.uqam.ca}
\urladdr{\url{http://sites.google.com/view/antoineabram}}

\author[J.~Bastidas]{Jose~Bastidas}
\address[Jose Bastidas]{}
\email{bastidas\_olaya.jose\_dario@uqam.ca}
\urladdr{\url{https://sites.google.com/view/bastidas}}

\newtheorem{mainthm}{Theorem}

\newtheorem{maincor}[mainthm]{Corollary}

\begin{document}

\begin{abstract}
We study the Ehrhart theory of hypersimplices of type C, as introduced by Lam and Postnikov for general crystallographic root systems. The $h^*$-polynomials of classical hypersimplices are known to relate to various Eulerian statistics on the symmetric group. In this paper, we introduce a new statistic and partial order on signed permutations, which we use to derive explicit formulas for the $h^*$-polynomials of type C hypersimplices. Additionally, we explore connections with other statistics, including flag-excedances and circular descents, flag-descents, and Coxeter descents.
\end{abstract}
\maketitle


\setcounter{tocdepth}{1}
\tableofcontents

\section{Introduction}
\label{s:intro}

Let $P \subset \RR^n$ be a $d$-dimensional polytope with vertices on a lattice $\Lambda \subset \RR^n$.
Ehrhart \cite{Ehrhart62} showed that the function
$
\defnn{L_P} = L^\Lambda_P : \ZZ_{\geq 0} \to \ZZ_{\geq 0} : r \mapsto | rP \cap \Lambda |
$
is a polynomial of degree $d$ in $r$.
Later, Stanley \cite{S80decomp} showed that the polynomial $\defnn{h_P^*(t)}$, defined from $L_P(r)$ by
\begin{equation}
\label{eq:def-h*}
\sum_{r \geq 0} L_P(r) t^r = \frac{h_P^*(t)}{(1-t)^{d+1}},
\end{equation}
has \emph{only} nonnegative integer coefficients.
The polynomial $\defnn{h_P^*(t)}$ is known as the \defn{(Ehrhart) $h^*$-polynomial} of~$P$.
We refer the reader to \cite[Chapter 3]{BR-discretely} for a comprehensive introduction to Ehrhart theory.

Understanding the coefficients of $L_P$ and $h^*_P$ for arbitrary polytopes $P$ is a difficult problem that has motivated substantial research in recent years; see, for instance, \cite{BJ24Ehalcoved,CL15Eh1,Ferroni21Eh2,Ferroni22Eh3,HJ16-flag,HLO15faceopen,J24Ehpositroid,SS18Lipschitz}.
However the sum of the coefficients of $h^*_P$ is simple to understand: it equals $\Vol(P)$, the normalized volume of $P$.

In this paper, we study the $h^*$-polynomial of \emph{type C hypersimplices}.

\subsubsection*{The classical (type A) case}

For positive integers $k < n$, the (type A) \defn{hypersimplex} $\defnn{\Delta_{n,k}} \subset \RR^n$ is the polytope whose vertices are all the $0/1$-vectors in $\RR^n$ with exactly $k$ ones.
Hypersimplices are fundamental objects in algebraic combinatorics; for example, they are the matroid polytopes of uniform matroids and also arise as the image of the Grassmannian under the moment map.

A classical result, implicit in the work of Laplace, states that the volume of the hypersimplex is given by the \defn{Eulerian numbers}. Specifically,
\begin{equation}
\label{eq:A-vol-Eul}
\Vol(\Delta_{n,k}) = \defnn{A_{n-1,k-1}} := | \{ w \in \mathfrak{S}_{n-1} \mid \des(w) = k-1 \} |,
\end{equation}
where $\defnn{\des(w)}$ denotes the \defn{descent statistic} (the number of positions $i$ such that $w_i > w_{i+1}$).
Thus, the coefficients of the $h^*$-polynomial of the hypersimplices \emph{refine} the Eulerian numbers.
Two decades ago, Katzman \cite{Katzman05} provided an explicit formula for these coefficients, though it was not manifestly positive. Only recently, Kim \cite{Kim20} proved a conjecture by Early \cite{Early17}, offering a combinatorial interpretation of these coefficients using a particular type of decorated set partitions. However, we are not aware of an interpretation of these coefficients via a joint distribution of two statistics on~$\mathfrak{S}_{n-1}$.

For the \defn{half-open hypersimplices} $\Delta'_{n,k}$, obtained from $\Delta_{n,k}$ by removing the facet lying in the hyperplane $x_n = 1$ when $k \geq 2$, Li \cite{Li12} provided the following two elegant, explicitly positive formulas for their $h^*$-polynomials using permutation statistics:
\begin{equation}
\label{eq:Li's}
h_{\Delta'_{n,k}}^*(t) = \sum_{\substack{w \in \mathfrak{S}_{n-1} \\ \des(w) = k-1 }} t^{{\rm cover}(w)}
\qquad\text{and}\qquad
h_{\Delta'_{n,k}}^*(t) = \sum_{\substack{w \in \mathfrak{S}_{n-1} \\ \exc(w) = k-1 }} t^{\des(w)}.
\end{equation}
Here $\defnn{\exc}$ denotes the \defn{excedance} statistic and $\defnn{{\rm cover}(w)}$ denotes the number of elements covered by $w$ in a certain partial order defined on $\mathfrak{S}_{n-1}$. This partial order is--up to a simple relabeling--exactly the order on cyclic permutations defined by Abram, Chapelier-Laget, and Reutenauer \cite{ACR21}. The second formula was originally conjectured by Stanley. Observe that, since $\des$ and $\exc$ have the same distribution, plugging $t = 1$ in either of the formulas above recovers \Cref{eq:A-vol-Eul}.

\subsubsection*{New results for type C hypersimplices}

In a series of influential papers, Lam and Postnikov \cite{LP07alcoved1,LP18alcoved2} define generalized hypersimplices for any crystallographic root system $\Phi$.
The \defn{$\Phi$-hypersimplex $\Delta_{\Phi,k}$} is determined by inequalities of the form $0 \leq \langle x , \alpha \rangle \leq 1$ for all simple roots $\alpha$ and $k-1 \leq \langle x , \theta \rangle \leq k$, where $\theta$ is the highest root of $\Phi$.

Explicitly, for the \defn{root system of type $C_n$} with simple roots $2e_1 \,,\, e_2-e_1\,,\,e_3-e_2,\dots,e_n-e_{n-1}$, these polytopes are
\begin{equation}
\label{eq:def-hypersimplex}
\defnn{\Delta_{C_n,k}} : = \{ x \in \RR^n \mid
0 \leq 2 x_1 \,,\, x_2 - x_1 \,,\, \dots \,,\, x_n - x_{n-1} \leq 1 \,\text{ and }\, k-1 \leq 2x_n \leq k \},
\end{equation}
for $k = 1 , \dots, 2n-1$.
The \defn{half-open hypersimplex $\Delta'_{C_n,k}$} is obtained by replacing the weak inequality $k-1 \leq 2x_n$ in \Cref{eq:def-hypersimplex} by the strict inequality $k-1 < 2x_n$; except at $k=1$, for which $\Delta'_{C_n,1} := \Delta_{C_n,1}$.
The vertices of $\Delta_{C_n,k}$ lie on the half-integer lattice ${\Lambda := \frac{1}{2}\ZZ^n}$, with respect to which we compute its Ehrhart and $h^*$-polynomial.

Let $\defnn{\hypoct_n}$ denote the group of \defn{signed permutations}: permutations $w$ of the set $\{-n,\dots,-1,1,\dots,n\}$ such that $w(-i) = - w(i)$. Our formulas are given in terms of statistics on the following subset of~$\hypoct_n$.

\begin{definition}
\label{def:setXn}
Let $\defnn{X_n} \subset \hypoct_n$ be the following collection of signed permutations:
\[
\defnn{X_n} := \{ w \in \hypoct_n \mid w^{-1}(1) > 0 \}.
\]
\end{definition}

Our first result relies on a shelling of the unimodular triangulation of $\Delta_{C_n,k}$ as an \defn{alcoved polytope}. Explicitly, we construct a partial order on $X_n$ whose cover relations are determined by the \defn{big ascent} statistic $\defnn{\basc}$ introduced in \cref{s:poset}. This poset serves as a combinatorial model for a certain interval of the \emph{weak order} on the alcoves of the affine reflection arrangement of type C.

\begin{mainthm}\label{t:cdes-basc}
For $n \geq 1$ and $k \geq 1$, the $h^*$-polynomial of the half-open hypersimplex $\Delta'_{C_n,k}$ is
\[
h_{\Delta'_{C_n,k}}^*(t) =
\sum_{ \substack{ w \in X_n : \\ \cdes(w^{-1}) = k}} t^{\basc(w)}.
\]
\end{mainthm}

The statistic $\defnn{\cdes}$ above is the \defn{circular descent statistic} of Lam and Postnikov \cite{LP18alcoved2}, which we review in \cref{ss:alcoved-hypersimplices}. By plugging $t = 1$ in \cref{t:cdes-basc}, we essentially recover their volume formula for generalized hypersimplices in \cite[Theorem 9.3]{LP18alcoved2}; see \Cref{eq:LP-volume}.

The half-open hypersimplices partition the \defn{fundamental parallelepiped}
\begin{equation}\label{eq:def-PiC}
\defnn{\Pi_{C_n}} : = \{ x \in \RR^n \mid
0 \leq 2 x_1 \,,\, x_2 - x_1 \,,\, \dots \,,\, x_n - x_{n-1} \leq 1 \}.
\end{equation}
Given that the fundamental parallelepiped and all the half-open hypersimplices have the same dimension, \cref{t:cdes-basc} implies the following formula for the $h^*$-polynomial of $\Pi_{C_n}$:
\begin{equation}\label{eq:h-parallel1}
h_{\Pi_{C_n}}^*(t) = \sum_k h_{\Delta'_{C_n,k}}^*(t) = \sum_{ w \in X_n } t^{\basc(w)}.
\end{equation}
The parallelepiped $(\Pi_{C_n},\frac{1}{2}\ZZ^n)$ is \emph{integrally equivalent} to the box $([0,1]\times[-1,1]^{n-1} , \ZZ^n)$.
By studying the triangulation obtained by slicing this box with the linear Coxeter arrangement of type $BC_n$, one finds an alternative formula for the $h^*$-polynomial of $\Pi_{C_n}$:
\begin{equation}\label{eq:h-parallel2}
h_{\Pi_{C_n}}^*(t) = \sum_{ w \in X_n } t^{\des_B(w)},
\end{equation}
where $\defnn{\des_B}$ denotes the usual descent statistic of $\hypoct_n$ as a Coxeter group.

A natural question arises by comparing \Cref{eq:h-parallel1,eq:h-parallel2}:
\begin{center}
\begin{minipage}{.7\linewidth}
\centering
\emph{Is there a formula for the $h^*$-polynomial of $\Delta'_{C_n,k}$
that uses the descent statistic of the hyperoctahedral group?}
\end{minipage}
\end{center}
The following result gives a positive answer to this question.

\begin{mainthm}
\label{t:fexc-desB}
For all $n \geq 1$ and $k \geq 1$, the $h^*$-polynomial of the half-open hypersimplex $\Delta'_{C_n,k}$ is
\[
h_{\Delta'_{C_n,k}}^*(t) = \sum_{ \substack{ w \in X_n : \\ \fexc(w) = k-1 } } t^{\des_B(w)},
\]
where $\defnn{\fexc}$ denotes the \defn{flag-excedance} statistic
of Foata and Han \cite{fh09signedV}.
\end{mainthm}

We prove this theorem in \cref{s:proofB}, by manipulation of generating functions. The definition of $\fexc$ and of the other statistics on signed permutations that appear in the formulas above are reviewed in \cref{ss:hypoct}.
Evaluating at $t = 1$, we find a new formula for the volume of the type C hypersimplices that is fundamentally different from that of Lam and Postnikov.

\begin{maincor}
\label{cor:volume-fexc}
For all $n \geq 1$ and $k \geq 1$,
\[
\Vol(\Delta_{C_n,k}) = |\{ w \in X_n \mid \fexc(w) = k-1 \}|.
\]
\end{maincor}

The statistic $\exc_B$ defined by $\exc_B(w) = \lfloor \tfrac{\fexc(w) + 1}{2} \rfloor$ (see \cite{B20polalg,BHS23PrimEul}) is Eulerian, meaning its distribution on $\hypoct_n$ is the same as that of $\des_B$. Using \cref{cor:volume-fexc}, we can recover the \defn{type B Eulerian numbers $B_{n,k}$} in terms of volumes of type C hypersimplices; c.f. \Cref{eq:A-vol-Eul}.

\begin{maincor}
\label{cor:volC-EulB}
For all $n \geq 1$ and $k \geq 1$, we have:
\begin{align*}
B_{n,k} &= \Vol(\Delta_{C_n,2k-1}) + 2 \Vol(\Delta_{C_n,2k}) + \Vol(\Delta_{C_n,2k+1})
\end{align*}
where $\Vol(\Delta_{C_n,j}) = 0$ whenever $j < 1$ or $j > 2n-1$.
\end{maincor}

In type A, the fundamental parallelepiped $\Pi_{A_{n+1}}$ is equivalent to the unit cube $[0,1]^n$, whose Ehrhart polynomial is $(r+1)^n$. Thus, Worpitzky's identity tells us that the $h^*$-polynomial of $\Pi_{A_{n+1}}$ is the Eulerian polynomial $E_{A_n}(t) = \sum_{k} A_{n,k} t^k$.
In \cref{s:Psi}, we obtain the corresponding result in type C by explicitly computing the Ehrhart series of $\Pi_{C_{n+1}}$.

\begin{mainthm}
\label{t:generating}
The following identity holds
\[
\sum_{n \geq 0} h^*_{\Pi_{C_{n+1}}}(t) \frac{x^n}{n!} = e^{3x(t-1)} {\rm Eul}_A(t,2x)^2,
\]
where ${\rm Eul}_A(t,x)$ is the exponential generating function of the Eulerian polynomials.
\end{mainthm}

As a byproduct of the proof of the formula above, we also deduce the real-rootedness of the $h^*$-polynomial of $\Pi_{C_n}$; see \cref{t:real-roots}.

In light of \Cref{eq:h-parallel1}, \cref{t:generating} can be interpreted as providing the generating function for the big ascent statistic on the sequence of posets $\{(X_n,\leq)\}_{n \geq 1}$. In \cref{s:limit}, we show that in a precise sense, the \emph{limit} of $X_n$ is the lattice $\defnn{\mathcal{SP}}$ of \defn{strict integer partitions}. This provides a type C analog of a result in \cite{ACR21}, where the authors show that the posets of cyclic permutations \emph{converge} to Young's lattice of integer partitions.

\subsubsection*{Related work}

Li's work on the $h^*$-polynomial of the (type A) half-open hypersimplices \cite{Li12} sparked significant interest in these polytopes. In \cite{HLO15faceopen}, Hibi, Li, and Ohsugi computed the face vectors of these polytopes. Han and Josuat-Vergès \cite{HJ16-flag} computed $h^*$-polynomials for slices of dilated hypercubes using flag statistics on colored permutations; their approach does not rely on unimodular triangulations or introduce new partial orders on the permutations. Sanyal and Stump \cite[\S 5]{SS18Lipschitz} extended Li's results to \emph{$P$-hypersimplices}. More recently, Jiang \cite{J24Ehpositroid}, Bullock and Jiang \cite{BJ24Ehalcoved}, have studied the Ehrhart theory of positroids and general alcoved polytopes, specifically in the case where $\Lambda$ is the integer lattice. Since the vertices of alcoved polytopes $P$ are only guaranteed to be rational, the functions $L^\Lambda_P$ are, in general, only quasi-polynomials.

\subsubsection*{Acknowledgments}

We would like to express our gratitude to Alejandro Morales for shedding light on the connections between the poset of circular permutations defined in \cite{ACR21} and the work of Lam and Postnikov in \cite{LP07alcoved1}. His insight, drawn from his ongoing collaboration with Benedetti, González D'Léon, Hanusa, and Yip, were instrumental in starting this project. Additionally, we thank Christophe Reutenauer and Nathan Chapelier-Laget for their valuable comments and stimulating discussions. We extend our gratitude to the members of LACIM for generously dedicating their time to review an earlier draft of this manuscript and for providing valuable feedback.

\section{Background}

We introduce the following notation that will be used throughout the document.
For a positive integer $n$, we write $\defnn{[n]} := \{1,\dots,n\}$ and $\defnn{\dbra{n}} := \{-n , \dots , -1 , 1 , \dots , n\}$. For brevity, we also write $\defnn{\ol{i}} := - i$ for all integers $i$. Similarly, for $i \in [n]$, we write $\defnn{e_{\ol{i}}} := - e_i$, where $\{e_i\}_{i \in [n]}$ denotes the canonical basis of $\RR^n$.

\subsection{Signed permutations}
\label{ss:hypoct}

A \defn{signed permutation} of length $n$ is a permutation $w$ of the set $\dbra{n}$ that satisfies $w(\ol{i}) = \ol{w(i)}$ for all $i \in \dbra{n}$.
Under composition, they form the \defn{hyperoctahedral group}, denoted $\defnn{\hypoct_n}$.
This is the finite Coxeter group of type $B_n$ or $C_n$.

The \defn{window notation} of a signed permutation $w \in \hypoct_n$ is the word $w_1 w_2 \dots w_n \in \dbra{n}^n$, where $w_i = w(i)$. We usually identify $w$ with its window notation and write $w = w_1 w_2 \dots w_n$.
On the other hand, the \defn{complete notation} of $w \in \hypoct_n$ is the word
\[
\defnn{\wt{w}}:=w_{\ol{n}}\cdots w_{\ol{1}} w_1\cdots w_n,
\]
where $w_{\ol{i}} = \ol{w_i} = w(\ol{i})$ for all $i \in \dbra{n}$.

\subsubsection{Signed permutation statistics and type B Eulerian numbers}

We review some well-studied statistics on signed permutations.

For a signed permutation $w \in \hypoct_n$, a \defn{descent} is a position $i \in [0,n-1]$ such that $w_i > w_{i+1}$, where by convention $w_0 := 0$. We let $\defnn{\des_B(w)}$ denote the number of descents of $w$. The statistic $\des_B$ is precisely the descent statistic of $\hypoct_n$ as a Coxeter group; see, for instance, \cite[Chapter 13]{petersen15}. The \defn{type B Eulerian numbers} enumerate signed permutations by the descent statistic; namely, for $0 \leq k \leq n$,
\[
\defnn{B_{n,k}} := |\{ w \in \hypoct_n \mid \des_B(w) = k \}|.
\]

The \defn{flag-descent} statistic, introduced by Adin, Brenti, and Roichman in \cite{abr01desmaj}, is a variation of the descent statistic, defined as:
\[
\defnn{\fdes(w)} := 2 \des_B(w) - \delta_{w_1 < 0},
\]
where $\delta_{w_1 < 0} \in \{0,1\}$ equals $1$ if and only if $w_1 < 0$. Crucially, and directly from this definition, we have that
\begin{equation}
\label{eq:desBandfdes}
\des_B(w) = \left\lfloor \tfrac{\fdes(w) + 1}{2} \right\rfloor,
\end{equation}
for all $w \in \hypoct_n$.

Additionally, we say $i \in [n]$ is an \defn{excedance} (respectively, \defn{negation}) of $w \in \hypoct_n$ if $w_i > i$ (respectively, $w_i < 0$). We denote the number of excedances and negations of $w$ by $\defnn{\exc(w)}$ and $\defnn{\fneg(w)}$, respectively. The \defn{flag-excedance} statistic, introduced by Foata and Han in \cite{fh09signedV} is defined by
\[
\defnn{\fexc(w)} := 2 \exc(w) + \fneg(w).
\]
In the same article, they show that the statistics $\fdes$ and $\fexc$ have the same distribution.
Consequently, the statistic
\begin{equation}
\label{eq:excB}
\defnn{\exc_B(w)} := \left\lfloor \tfrac{\fexc(w) + 1}{2} \right\rfloor,
\end{equation}
introduced by the second author in \cite{B20polalg}, is Eulerian (has the same distribution as $\des_B$.)

\subsection{Alcoved polytopes}
\label{s:alcoved}

This section provides an overview of key aspects of the work by Lam and Postnikov in \cite{LP07alcoved1,LP18alcoved2} that are central to the discussion in this paper. We assume some familiarity with the theory of Coxeter groups and root systems. For terminology, we refer the reader to the books by Bourbaki \cite{Bour1968} and by Humphreys \cite{Humphreys}.

\subsubsection{Root systems}

Consider an $n$-dimensional real Euclidean space $V$ with a nondegenerate inner product $\langle \cdot , \cdot \rangle$. Let $\Phi \subset V$ be an irreducible crystallographic root system with a fixed set of \defn{simple roots} $\{\alpha_1,\dots,\alpha_n\} \subset \Phi$. Denote by $\defnn{\Phi^+} := \Phi \cap \ZZ_{\geq 0}\{\alpha_1,\dots,\alpha_n\}$ its set of \defn{positive roots}.
The \defn{root order} is the partial order on $\Phi^+$ determined by $\alpha \leq \beta$ if and only if $\beta - \alpha \in \ZZ_{\geq 0}\{\alpha_1,\dots,\alpha_n\}$. It has a unique maximum element $\defnn{\theta}$, called the \defn{highest root} of $\Phi$. Let $a_1,\dots,a_n \in \ZZ_{\geq 0}$ be the constants such that
\begin{equation}
\label{eq:highest}
\theta = a_1 \alpha_1 + \dots + a_n \alpha_n.
\end{equation}
The \defn{Coxeter number} of $\Phi$ is $h(\Phi) := a_1 + \dots + a_n + 1$. Finally, the \defn{fundamental coweights} $\{\omega_1,\dots,\omega_n\}$ form the basis of $V$ dual to the simple roots, so $a_i = \langle \theta , \omega_i \rangle$.

\begin{example}[Type A]
Let $V = \RR^{n+1}/\RR {\bf 1}$, where ${\bf 1}$ denotes the all-ones vector.
The root system $\Phi \subset V$ of \defn{type $A_n$} has simple roots
\[
\alpha_1 = e_2 - e_1 , \quad
\alpha_2 = e_3 - e_2 , \quad \dots \quad
\alpha_{n-1} = e_n - e_{n-1}, \quad
\alpha_n = e_{n+1} - e_n.
\]
Its set of positive roots and its highest root are respectively
\[
\Phi^+ = \{ e_j - e_i \mid 1 \leq i < j \leq n+1 \}
\quad\text{and}\quad
\theta = \alpha_1 + \alpha_2 + \dots + \alpha_n = e_{n+1} - e_1.
\]
Lastly, its fundamental coweights are (the projections to $V$ of)
\[
\omega_1 = e_2 + e_3 + \dots + e_{n+1}, \quad
\omega_2 = e_3 + \dots + e_{n+1}, \quad
\dots \quad
\omega_n = e_{n+1}.
\]
\end{example}

\begin{example}[Type C]
\label{ex:roots-C}
Let $V = \RR^n$.
The root system $\Phi \subset V$ of \defn{type $C_n$} has simple roots
\[
\alpha_1 = 2e_1, \quad
\alpha_2 = e_2 - e_1 , \quad \dots \quad
\alpha_{n-1} = e_{n-1} - e_{n-2} , \quad
\alpha_n = e_n - e_{n-1}.
\]
Its set of positive roots and its highest root are respectively
\[
\Phi^+ = \{ e_j \pm e_i \mid 1 \leq i < j \leq n \} \cup \{ 2e_i \mid i \in [n] \}
\quad\text{and}\quad
\theta = \alpha_1 + 2\alpha_2 + \dots + 2\alpha_{n-1} + 2\alpha_n = 2 e_n.
\]
\cref{f:H-Cn} shows the positive roots for $n=2$.
Lastly, its fundamental coweights are
\[
\omega_1 = \tfrac{1}{2} \left(e_1 + e_2 + \dots + e_n \right), \quad
\omega_2 = e_2 + e_3 + \dots + e_n, \quad
\omega_3 = e_3 + \dots + e_n, \quad
\dots \quad
\omega_n = e_n.
\]
\end{example}

\subsubsection{Affine Coxeter arrangement and alcoves}

The \defn{affine Coxeter arrangement} associated with $\Phi$ is the collection of hyperplanes
\[
\defnn{\HH_{\Phi}} := \{ H_{\alpha,k} \mid \alpha \in \Phi^+ , k \in \ZZ \},
\qquad\text{where}\qquad
\defnn{H_{\alpha,k}} := \{x \in V \mid \angbra{x,\alpha} = k\}.
\]
The \defn{affine Coxeter group} $\defnn{\wt{W}_\Phi}$ is the group of affine transformations of $V$ generated by reflections across the hyperplanes in $\HH_{\Phi}$.
The \defn{finite Coxeter group} $\defnn{W_\Phi}$ is the subgroup of $\defnn{\wt{W}_\Phi}$ generated by reflections across the linear hyperplanes $H_{\alpha,0}$.
See \cref{f:H-Cn} for a picture of $\HH_{\Phi}$ when $\Phi$ is of type $C_2$.

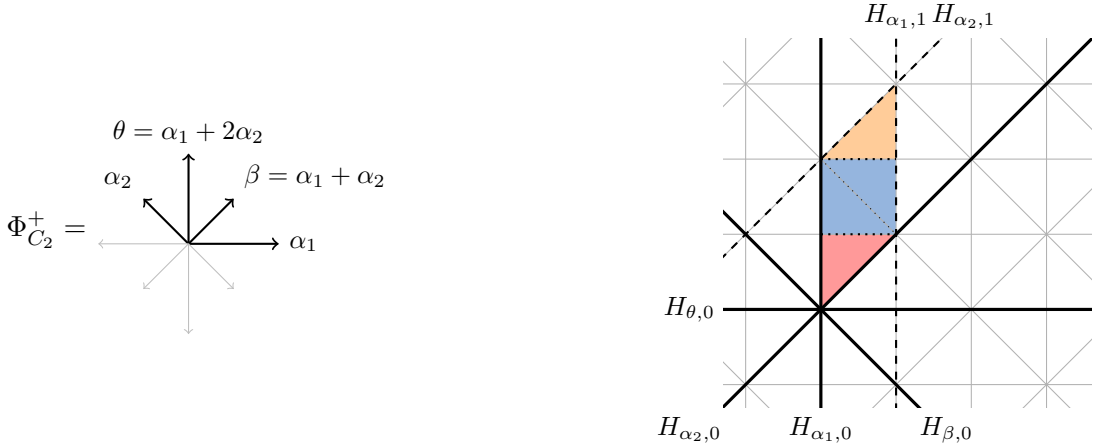
\begin{figure}[h]
\[
\Phi^+_{C_2} =
\begin{gathered}
\begin{tikzpicture}[scale = .6]
\def\B{1}
\draw[thick,->] (0,0) -- ( 2,0) node[right] {\small $\alpha_1$};
\draw[thick,->] (0,0) -- ( 1,1) node[above right] {\small $\beta = \alpha_1 + \alpha_2$};
\draw[thick,->] (0,0) -- ( 0,2) node[above] {\small $\theta = \alpha_1 + 2\alpha_2$};
\draw[thick,->] (0,0) -- (-1,1) node[above left] {\small $\alpha_2$};
\draw[->, lightgray] (0,0) -- (-2,0) {};
\draw[->, lightgray] (0,0) -- (-1,-1) {};
\draw[->, lightgray] (0,0) -- (0,-2) {};
\draw[->, lightgray] (0,0) -- (1,-1) {};
\end{tikzpicture}
\end{gathered}
\hspace*{.2\linewidth}
\begin{gathered}
\begin{tikzpicture}[scale = 2]
\def\B{1}
\def\p{.15}
\node at (-\B-2.5*\p,-\B/2) {\small $H_{\theta,0}$};
\node at ( -\B/2,-\B-2*\p) {\small $H_{\alpha_1,0}$};
\node at (-\B-2.5*\p,-\B-2*\p) {\small $H_{\alpha_2,0}$};
\node at ( \B/3,-\B-2*\p) {\small $H_{\beta,0}$};
\node at ( 0,\B+3*\p) {\small $H_{\alpha_1,1}$};
\node at (3*\p,\B+3*\p) {\small $H_{\alpha_2,1}$};
\clip (-\B-\p,-\B-\p) rectangle (\B+2*\p,\B+2*\p);
\path[fill=white] (-\B-\p,-\B-\p) rectangle (\B+2*\p,\B+2*\p);
\path[fill=red!40] (-.5,-.5) --++ (0,.5) --++ (.5,0) -- cycle;
\path[fill=RoyalBlue!40] (-.5,0) --++ (0,.5) --++ (.5,0) --++ (0,-.5) -- cycle;
\path[fill=orange!40] (-.5,.5) --++ (.5,0) --++ (0,.5) -- cycle;
\foreach \y in {-2\B,...,2\B}{
\draw[black!30,shift = {(0,\y/2)}] (-2\B,0) --++ (4\B,0);
\draw[black!30,shift = {(\y/2,0)}] (0,-2\B) --++ (0,4\B);
\draw[black!30,shift = {(\y,0)}] (-2\B,-2\B) --++ (4\B,4\B);
\draw[black!30,shift = {(0,\y)}] (-2\B,2\B) --++ (4\B,-4\B);
}
\draw[very thick] (-.5,-2\B) --++ (0,4\B);
\draw[very thick] (-2\B,-.5) --++ (4\B,0);
\draw[very thick] (-2\B,-2\B) --++ (4\B,4\B);
\draw[very thick] (-2\B,2\B-1) --++ (4\B,-4\B);
\draw[dashed,thick] (0,-2\B) --++ (0,4\B);
\draw[dashed,thick] (-2\B,-2\B+1) --++ (4\B,4\B);
\draw[dotted,thick] (-.5,0) --++ (.5,0);
\draw[dotted,thick] (-.5,.5) --++ (.5,0);
\draw[dotted] (-.5,.5) --++ (.5,-.5);
\end{tikzpicture}
\end{gathered}
\]
\caption{
Left: Positive roots of type $C_2$.
Right: The affine arrangement $\HH_{\Phi_{C_2}}$
with the fundamental parallelepiped $\Pi_{C_2}$ subdivided into the three hypersimplices $\Delta_{C_2,1}$ (red, bottom), $\Delta_{C_2,2}$ (blue, center), and $\Delta_{C_2,3}$ (orange, top).
Only thick or dashed hyperplanes are labeled.
}
\label{f:H-Cn}
\end{figure}

The complement in $V$ of the hyperplanes in $\HH_{\Phi}$ is the disjoint union of bounded convex open sets, with their closures being the \defn{alcoves} of the arrangement. The \defn{fundamental alcove} is
\begin{equation}
\label{eq:fund-alcove}
\begin{split}
\defnn{A_\circ} := & \{ x \in V \mid 0 \leq \langle x , \alpha \rangle \leq 1 \text{ for all } \alpha \in \Phi^+ \} \\
= & {\rm Conv}\{ 0 , \tfrac{\omega_1}{a_1} , \dots , \tfrac{\omega_n}{a_n}\}.
\end{split}
\end{equation}
Let $\leq$ denote the \defn{weak order} on the alcoves of $\HH_\Phi$ with base region $A_\circ$ \cite{mandel82thesis,edelman84regions}.
That is, $\defnn{A \leq A'}$ if and only if every hyperplane $H \in \HH_\Phi$ separating $A_\circ $ and $A$ also separates $A_\circ $ and $A'$. The group ${\wt{W}_\Phi}$ acts faithfully and transitively on the collection of alcoves, allowing us to identify elements $w \in \wt{W}_\Phi$ with alcoves $w \cdot A_\circ$. Under this identification, the weak order on alcoves and the right weak order on $\wt{W}_\Phi$ as a Coxeter group agree.

The \defn{fundamental parallelepiped} is the polytope
\begin{equation}
\label{eq:def-fund_parall}
\begin{split}
\defnn{\Pi_\Phi} := & \{ x \in V \mid 0 \leq \langle x , \alpha_i \rangle \leq 1 \text{ for all } i = 1,\dots,n \} \\
= & [0,1]\omega_1 + [0,1]\omega_2 + \dots + [0,1]\omega_n,
\end{split}
\end{equation}
where $[0,1]\omega_i \subset V$ is the segment $\{ \lambda \omega_i \mid 0 \leq \lambda \leq 1 \}$ and the sum denotes Minkowski sum:
\[ A + B = \{ a + b \mid a \in A,\, b \in B\}. \]
It is bounded by hyperplanes of the affine arrangement $\mathcal{H}_\Phi$, and is therefore a union of its alcoves. In fact, it contains exactly
\begin{equation}
\label{eq:volPi}
\Vol(\Pi_\Phi)/\Vol(A_\circ) = n! \cdot a_1 \dots a_n
\end{equation}
alcoves, where $a_1,\dots,a_n$ are the constants in \Cref{eq:highest}. Moreover, the alcoves contained in $\Pi_\Phi$ form an interval in the weak order with minimum element $A_\circ$. We let
\begin{equation}
\label{eq:defI}
\defnn{\mathcal{I}_\Phi} := \{ A \text{ alcove } \mid A \subset \Pi_\Phi \}
\end{equation}
denote this poset and
\begin{equation}
\label{eq:def-Psi}
\defnn{\Psi_\Phi(t)} := \sum_{A \in \mathcal{I}_\Phi} t^{{\rm cover}(A)}
\end{equation}
denote the polynomial whose coefficients keep the distribution of the number of lower covers of elements in $\mathcal{I}_\Phi$. When $\Phi$ is of type A, $\Psi_{A_n}(t) = E_{A_{n-1}}(t)$, the Eulerian polynomial; see for instance \cite[Theorem 27]{ACR21}. In \cite{GSS12looping}, Gashi, Schedler, and Speyer compute the Hilbert polynomial $\mathcal{I}_\Phi$ for the root systems of types A, B, C, and D.

\subsubsection{Alcoved polytopes and generalized hypersimplices}
\label{ss:alcoved-hypersimplices}

Every alcove is determined by a set of inequalities $k_\alpha \leq \langle x , \alpha \rangle \leq k_\alpha + 1$ for some vector of integers $(k_\alpha)_{\alpha \in \Phi^+} \in \ZZ^{\Phi^+}$.
If we relax the constraints on the integers $k_\alpha$ and $k_\alpha + 1$ being consecutive, we obtain a larger class of polytopes.

\begin{definition}[Lam--Postnikov]
A polytope $P \subset V$ is called an \defn{alcoved polytope} if it is defined by inequalities $a_\alpha \leq \langle x , \alpha \rangle \leq b_\alpha$ for some collection of integers\footnote{In their later work, Lam and Postnikov extend the definition of alcoved polytopes to include non-integer inequalities; see, for instance, \cite{LP24Polypositroids}. However, since we are interested in the \emph{alcove triangulation} of these polytopes, we focus on the integer case.} $\{a_\alpha,b_\alpha\}_{\alpha \in \Phi^+}$.
\end{definition}

If $P$ is a full-dimensional polytope, then $P$ is an alcoved polytope if and only if $P$ is the union of some alcoves of $\HH_{\Phi}$. This family of polytopes naturally includes the alcoves of $\HH_\Phi$ and the fundamental parallelepiped in \Cref{eq:def-fund_parall}. Another important family of alcoved polytopes can be obtained by slicing the fundamental parallelepiped with the hyperplanes of $\HH_\Phi$ orthogonal to the highest root $\theta$.

\begin{definition}[Lam--Postnikov]
The generalized hypersimplices, or \defn{$\Phi$-hypersimplices}, are the polytopes
\begin{equation}\label{eq:def-hyper}
\defnn{\Delta_{\Phi,k}} := \{ x \in \Pi_\Phi \mid k-1 \leq \langle x , \theta \rangle \leq k \},\quad \text{for } k = 1 , \dots, h(\Phi) - 1.
\end{equation}
\end{definition}

See \cref{f:H-Cn} for a picture of all $\Phi$-hypersimplices when the root system is of type $C_2$.
When $\Phi$ is the root system of type $A_{n-1}$, the hypersimplex $\Delta_{\Phi,k}$ is integrally equivalent to the usual hypersimplex $\Delta_{n,k}$ discussed in the \nameref{s:intro}, whose volume is determined by the distribution of the descent statistic on $\mathfrak{S}_{n-1}$. More generally, the volume of the \defn{$\Phi$-hypersimplices} is determined by the circular descent statistic $\cdes$ described below.

Set $\defnn{\alpha_0} := - \theta$, and for each $w \in W_\Phi$ and $i = 0,1,\dots,n$ let
\[
d_i(w) = \begin{cases}
0 & \text{if } w(\alpha_i) \in \Phi^+,\\
1 & \text{if } w(\alpha_i) \in \Phi^-.
\end{cases}
\]

\begin{definition}[Lam--Postnikov]
\label{def:LP-cdes}
The \defn{circular descent number} of $w$ is
\[
\defnn{\cdes(w)} := d_0(w) + a_1 d_1(w) + a_2 d_2(w) + \dots + a_n d_n(w),
\]
where $\{a_i\}_i$ are the integers such that $\theta = a_1 \alpha_1 + \dots + a_n \alpha_n$.
\end{definition}

Recall that the \defn{index of connection} is the constant $\defnn{f} \in \ZZ$ such that $|W_\Phi| = f \cdot n! \cdot a_1 \dots a_n$. The following formula for the volume of the $\Phi$-hypersimplices is \cite[Theorem 9.3]{LP18alcoved2}.

\begin{equation}
\label{eq:LP-volume}
\Vol(\Delta_{\Phi,k}) = \frac{1}{f} \left| \{ w \in W_\Phi \mid \cdes(w) = k \} \right|.
\end{equation}

\section{Type C hypersimplices}
\label{ss:our-case}

For the rest of the document, we will focus on the case where $\Phi = \Phi_{C_n}$ is the root system of type $C_n$ in \cref{ex:roots-C}. Recall that its simple roots are
\[
\alpha_1 = 2 e_1 ,\quad
\alpha_2 = e_2 - e_1 ,\quad
\dots, \quad
\alpha_{n-1} = e_{n-1} - e_{n-2},\quad\text{and}\quad
\alpha_n = e_n - e_{n-1}.
\]
The corresponding set of positive roots is
\[
\Phi^+ = \{ e_j \pm e_i \mid 1 \leq i < j \leq n \} \cup \{ 2 e_i \mid 1 \leq i \leq n \},
\]
and the highest root is
\[
\theta = 1 \alpha_1 + 2 \alpha_2 + \dots + 2 \alpha_{n-1} + 2 \alpha_n = 2 e_n.
\]
Hence, $h(\Phi_{C_n}) = 1 + 2 + \dots + 2 + 1 = 2n$.
Its fundamental coweights are
\[
\omega_1 = \tfrac{1}{2}(e_1 + e_2 + \dots + e_n) ,\quad
\omega_2 = e_2 + \dots + e_n ,\quad
\dots, \quad
\omega_{n-1} = e_{n-1} + e_n ,\quad
\omega_n = e_n.
\]
Therefore, the lattice generated by the vertices of $A_\circ$ is $\Lambda = \ZZ\{ \frac{\omega_i}{a_i} \}_i = \frac{1}{2} \ZZ^n$; which contains the vertices of all the alcoves of $\HH_{\Phi_{C_n}}$.
Every alcove is thus \defn{unimodular} with respect to $\Lambda$, in particular they all have normalized volume $1$. This also occurs in type~A, but not in type B or D; see \cref{s:BandD}.
The associated finite Coxeter group is the hyperoctahedral group $W_{\Phi_C} \cong \hypoct_n$ reviewed in \cref{ss:hypoct}.

The \defn{type C fundamental parallelepiped $\Pi_{C_n}$} is the polytope
\[
\defnn{\Pi_{C_n}} := \{ x \in \RR^n \mid 0 \leq 2 x_1, x_2 - x_1 , \dots , x_n - x_{n-1} \leq 1 \}.
\]
For $1 \leq k \leq 2n-1$, the \defn{type C hypersimplex $\Delta_{C_n,k}$} is
\[
\defnn{\Delta_{C_n,k}} := \{ x \in \Pi_n \mid k-1 \leq 2 x_n \leq k \}.
\]

\begin{definition}
\label{def:half-open-C}
The \defn{half-open type C hypersimplices $\Delta'_{C_n,k}$} are defined by
\[
\defnn{\Delta'_{C_n,1}} := \{ x \in \Pi_{C_n} \mid 0 \leq 2 x_n \leq 1 \} = \Delta_{C_n,1},
\]
and, for $2 \leq k \leq 2n-1$,
\begin{equation}
\label{eq:def-half-open}
\defnn{\Delta'_{C_n,k}} := \{ x \in \Pi_{C_n} \mid k-1 < 2 x_n \leq k \}
= \Delta_{C_n,k} \setminus (\Delta_{C_n,k} \cap \Delta_{C_n,k-1}).
\end{equation}
\end{definition}

\section{Proof of Theorem \ref{t:fexc-desB}}
\label{s:proofB}

We prove \cref{t:fexc-desB} using generating functions and adapting techniques of Li \cite{Li12} to our context. In particular, we need to understand the generating function for: \emph{i.} the joint distribution of the statistics $\fexc$ and $\des_B$ on the sets $\{X_n\}_{n \geq 1}$ of signed permutations $w$ satisfying $w^{-1}(1) > 0$ (see \cref{def:setXn}), and \emph{ii.} the values of the Ehrhart polynomial of all type C half-open hypersimplices. We achieve these objectives in the following two propositions, respectively.

\begin{proposition}\label{p:aux1}
The following identity holds.
\begin{equation}\label{eq:p-aux1}
\sum_{n \geq 1} \sum_{w \in X_n} s^{\fexc(w)} t^{\des_B(w)} \dfrac{u^n}{(1-t)^{n+1}} =
\dfrac{1}{(1+s)} \sum_{r \geq 0} \dfrac{(1-us^2)^{r+1} - (1-u)^{r+1}}{(1-u)^{r+1}-s(1-u)(1-us^2)^r} t^r.
\end{equation}
\end{proposition}

\begin{proposition}\label{p:aux2}
For every $r \geq 0$,
\begin{equation}\label{eq:p-aux2}
\sum_{n \geq 1} \sum_{k \geq 0} L_{\Delta'_{C_n,k+1}}(r) s^k u^n =
\dfrac{(1-us^2)^{r+1} - (1-u)^{r+1}}{(1+s) \big( (1-u)^{r+1}-s(1-u)(1-us^2)^r \big)}.
\end{equation}
\end{proposition}

We give a complete proof of these two results below, but we first use them to prove one of our main results.

\begin{proof}[{Proof of \cref{t:fexc-desB}}]
Multiply both sides of \Cref{eq:p-aux2} by $t^r$ and sum over all $r \geq 0$.
The resulting right-hand side is exactly the right-hand side of \Cref{eq:p-aux1}. Therefore,
\[
\sum_{n \geq 1} \sum_{k \geq 0} \sum_{r \geq 0}
L_{\Delta'_{C_n,k+1}}(r) t^r s^k u^n
=
\sum_{n \geq 1} \sum_{w \in X_n}
s^{\fexc(w)} t^{\des_B(w)} \dfrac{u^n}{(1-t)^{n+1}}.
\]
Extracting the coefficient of $s^{k-1} u^n$ on both sides of this equation, we obtain
\[
\sum_{r \geq 0} L_{\Delta'_{C_n,k}}(r) t^r =
(1-t)^{-(n+1)} \sum_{\substack{ w \in X_{n} \\ \fexc(w) = k -1}} t^{\des_B(w)}.
\]
The result now follows by the definition of the $h^*$-polynomial (\Cref{eq:def-h*}) and the fact that $\Delta'_{C_n,k}$ has dimension $n$.
\end{proof}

We now proceed to prove the two propositions in detail.

\begin{proof}[{Proof of \cref{p:aux1}}]
Foata and Han \cite[Equation (9.3)]{fh09signedV} show that
\begin{multline*}
\sum_{n \geq 0} \sum_{w \in \hypoct_n} s^{\fexc(w)} t^{\fdes(w)} \dfrac{u^n}{(1-t^2)^n} = \\
\sum_{r \geq 0} \bigg( \dfrac{(1-s)(1-t)t^{2r}}{(1-u)^{r+1}(1-us^2)^{-r}-s(1-u)}
- \dfrac{(1-s)(1-t)t^{2r+1}}{(1-u)^{r+1}(1-us^2)^{-r}-s(1-us)} \bigg).
\end{multline*}
We will first manipulate this expression to replace the statistic $\fdes$ with $\des_B$, and then to restrict the domain of the internal sum from all signed permutations to those for which $w^{-1}(1) > 0$.

Recall from \Cref{eq:desBandfdes} that $\des_B(w) = \floor{\tfrac{\fdes(w)+1}{2}}$ for all $w \in \hypoct_n$.
Since the factors $\frac{1}{(1-t^2)^n}$ only involve even powers of~$t$,
we can expand the expression above in powers of~$t$ and substitute $t^{2k-1},t^{2k} \mapsto t^k$
to obtain the generating function for the statistics $(\fexc,\des_B)$.
For the two terms inside the sum of the right-hand side, these substitutions are, respectively,
\begin{align*}
(1-t)t^{2r} = t^{2r} - t^{2r+1} &\mapsto t^r - t^{r+1} = (1-t)t^r , & \text{and} \\
(1-t)t^{2r+1} = t^{2r+1} - t^{2r+2} & \mapsto t^{r+1} - t^{r+1} = 0.
\end{align*}
Therefore, the generating function for the joint distribution of $(\fexc,\des_B)$ is
\begin{equation}
\label{eq:FH}
\sum_{n \geq 0} \sum_{w \in \hypoct_n} s^{\fexc(w)} t^{\des_B(w)} \dfrac{u^n}{(1-t)^{n+1}}
= \sum_{r \geq 0} \dfrac{(1-s) (1-us^2)^r t^r}{(1-u)^{r+1}-s(1-u)(1-us^2)^r}.
\end{equation}
Notice that we have divided both sides by $(1-t)$ and multiplied the right-hand side by $1 = \frac{(1-us^2)^r}{(1-us^2)^r}$.

Now, for $n \geq 1$ and $k \in \dbra{n}$, let
\[
B^{(k)}_n(s,t) = \sum_{ \substack{ w \in \hypoct_n \\ w^{-1}(1) = k } } s^{\fexc(w)} t^{\des_B(w)}.
\]
Also, let $B_0(s,t) = 1$ and for $n > 0$,
\[
B_n(s,t) = \sum_{k \in \dbra{n}} B^{(k)}_n(s,t),
\qquad \text{and}\qquad
B^{>1}_n(s,t) = \sum_{k > 1} B^{(k)}_n(s,t).
\]
Thus, \cref{eq:FH} is the ordinary generating function of $\frac{B_n(s,t)}{(1-t)^{n+1}}$,
and we want to compute the ordinary generating function of $\frac{B^{(1)}_n(s,t) + B^{>1}_n(s,t)}{(1-t)^{n+1}}$.

Observe that:
\begin{itemize}
\item If $w = 1 w_2 \dots w_n$ and $u$ is the standardization of $w_2 \dots w_n$
(that is, $|u_i| = |w_{i+1}|-1$ and $u_i w_{i+1} > 0$),
then $\fexc(u) = \fexc(w)$ and $\des_B(u) = \des_B(w)$.
Indeed, $1$ was neither an excedance nor a negation of $w$ and $i$ is an excedance (resp. negation) of $u$ if and only if $i+1$ is an excedance (resp. negation) of $w$. Similarly, since $w_1 > 0$, $0$ is not a descent of $w$ and $i \in [0,n-2]$ is a descent of $u$ if and only if $i+1$ is a descent of $w$.
Thus,
\[
B^{(1)}_n(s,t) = B_{n-1}(s,t).
\]

\item If $w = 1 w_2 \dots w_n$ and $u = \ol{1} w_2 \dots w_n$, then
$\fexc(u) = \fexc(w) + 1$ (since $1$ is now a negation of $u$) and
$\des_B(u) = \des_B(w) + 1$ (since $0$ is now a descent of $u$).
Thus,
\[
B^{(-1)}_n(s,t) = st B^{(1)}_{n}(s,t).
\]

\item If $w^{-1}(1) = k > 1$ and $u$ is obtained from $w$ by negating $w_k = 1$, then
$\fexc(u) = \fexc(w) + 1$ (since $k$ is now a negation of $u$) and
$\des_B(u) = \des_B(w)$ (since $w_{k-1}$ and $w_{k+1}$ are either both larger than $1$ and than $-1$, or both smaller than $1$ and than $-1$).
Thus,
\[
B^{(-k)}_n(s,t) = s B^{(k)}_{n}(s,t), \qquad\text{for all } k > 1.
\]
\end{itemize}
Therefore, $B_n(s,t) = (1 + st) B_{n-1}(s,t) + (1 + s) B^{>1}_n(s,t)$, and
\[
B^{(1)}_n(s,t) + B^{>1}_n(s,t)
= B_{n-1}(s,t) + \dfrac{B_n(s,t) - (1+st)B_{n-1}(s,t)}{(1+s)}
= \dfrac{B_n(s,t) + s(1-t)B_{n-1}(s,t)}{(1+s)}.
\]
Using the generating function for $\frac{B_n(s,t)}{(1-t)^{n+1}}$ in \Cref{eq:FH}, we get
\begin{align*}
\sum_{n \geq 1} (B^{(1)}_n(s,t) + B^{>1}_n(s,t)) \dfrac{u^n}{(1-t)^{n+1}}
=& \dfrac{1}{(1+s)} \left( \sum_{n \geq 1} B_n(s,t) \dfrac{u^n}{(1-t)^{n+1}} +
su \sum_{n \geq 1} B_{n-1}(s,t) \dfrac{u^{n-1}}{(1-t)^n} \right) \\
=& \dfrac{1}{(1+s)} \left( \sum_{r \geq 0} \dfrac{(1+su) (1-s) (1-us^2)^r t^r}{(1-u)^{r+1}-s(1-u)(1-us^2)^r}
- \dfrac{1}{(1-t)} \right) \\
=& \dfrac{1}{(1+s)} \sum_{r \geq 0} \bigg( \dfrac{(1 +su) (1-s) (1-us^2)^r}{(1-u)^{r+1}-s(1-u)(1-us^2)^r} - 1 \bigg)t^r \\
=& \dfrac{1}{(1+s)} \sum_{r \geq 0} \dfrac{(1-us^2)^{r+1} - (1-u)^{r+1}}{(1-u)^{r+1}-s(1-u)(1-us^2)^r} t^r,
\end{align*}
where in the last step we use that
$
(1+us)(1-s)(1-us^2)^r - s(1-u)(1-us^2)^r = (1-us^2)^{r+1}.
$
\end{proof}

\begin{proof}[{Proof of \cref{p:aux2}}]
Recall that the hypersimplex $\Delta_{C_n,k+1}$ is determined by inequalities
\[
0 \leq 2 x_1 \,,\, x_2 - x_1 \,,\, \dots \,,\, x_n - x_{n-1} \leq 1 \quad\text{ and }\quad k \leq 2x_n \leq k+1.
\]
Under the change of coordinates
\[
y_1 = 2 x_1 ,\ y_2 = 2(x_2 - x_1) ,\ \dots, \ \ y_n = 2(x_n - x_{n-1}),
\]
counting half-integer points in dilations of $\Delta_{C_n,k+1}$
corresponds to counting integer points inside dilations of
\[
\{ y \in \RR^n \mid 0 \leq y_1 \leq 1 \,,\:\: 0 \leq y_2,\dots,y_n \leq 2 \,,\:\: k \leq y_1 + \dots + y_n \leq k+1 \}.
\]
Therefore, for any $k > 0$ and $r \in \ZZ_{\geq 0}$,
\begin{multline*}
L_{\Delta'_{C_n,k+1}}(r) = \Big| \{ y \in \ZZ^n \mid 0 \leq y_1 \leq r \,,\:\: 0 \leq y_2,\dots,y_n \leq 2r \,,\:\: kr < y_1+\dots+y_n\leq (k+1)r\} \Big|.
\end{multline*}
For $k = 0$, we need to add $1$ to the formula above to account for the point $(0,\dots,0) \in \Delta'_{C_n,1} = \Delta_{C_n,1}$.
Thus,
\begin{align*}
L_{\Delta'_{C_n,k+1}}(r)
=& \delta_{k,0} + ([x^{kr+1}]+\cdots+[x^{(k+1)r}]) \left(\frac{1-x^{r+1}}{1-x}\right) \left(\frac{1-x^{2r+1}}{1-x}\right)^{n-1} \\
=& \delta_{k,0} + [x^{kr}] (x^{-1}+\cdots+x^{-r}) \left(\frac{1-x^{r+1}}{1-x}\right) \left(\frac{1-x^{2r+1}}{1-x}\right)^{n-1} \\
=& \delta_{k,0} + [x^{kr}] \frac{(1-x^r)(1-x^{r+1})(1-x^{2r+1})^{n-1}}{(1-x)^{n+1}x^r},
\end{align*}
where $\delta_{k,0}$ denotes the Kronecker delta and $[x^j] f(x)$ denotes the coefficient of $x^j$ in the power series expansion of $f(x)$.
Therefore,
\begin{align*}
\sum_{n \geq 1} L_{\Delta'_{C_n,k+1}}(r) u^n
=& \sum_{n \geq 1} \delta_{k,0} u^n + \sum_{n \geq 1}[x^{kr}]\frac{(1-x^r)(1-x^{r+1})(1-x^{2r+1})^{n-1} }{(1-x)^{n+1}x^r}u^n\\
=& \delta_{k,0}\dfrac{u}{1-u} + [x^{kr}] \frac{(1-x^r)(1-x^{r+1})u}{(1-x)^2 x^r}\sum_{n \geq 1}\left(\frac{(1-x^{2r+1})u}{1-x}\right)^{n-1}\\
=& \delta_{k,0}\dfrac{u}{1-u} + [x^{kr}] \frac{(1-x^r)(1-x^{r+1})u}{(1-x) x^r (1-x-u+ux^{2r+1})}.
\end{align*}
So,
\[
\sum_{n \geq 1} \sum_{k \geq 0} L_{\Delta'_{C_n,k+1}}(r) s^k u^n =
\dfrac{u}{1-u} + \sum_{k \geq 0} s^k [x^{kr}] \frac{(1-x^r) (1-x^{r+1}) u}{(1-x) x^r (1-x-u+ux^{2r+1})}.
\]
Let $z = \frac{1-ux^{2r}}{1-u}$.
We have,
\begin{equation}\label{eq:long-proof-aux1}
\sum_{n \geq 1} \sum_{k \geq 0} L_{\Delta'_{C_n,k+1}}(r) s^k u^n
= \dfrac{u}{1-u} + \sum_{k \geq 0} s^k [x^{kr}] \frac{(1-x^r) (1-x^{r+1})u}{x^r (1-x) (1-u) (1-zx)}.
\end{equation}
Since
\[
\sum_{k \geq 0} s^k [x^{kr}] x^m = \begin{cases}
s^j & \text{if } m = rj \text{ for some } j, \\
0 & \text{otherwise,}
\end{cases}
\]
we can drop all the terms of the form $c_m x^m$ where $m$ is not a multiple of $r$ in the expansion of $\frac{(1-x^r) (1-x^{r+1})u}{x^r (1-x) (1-u) (1-zx)}$ above.
Hence, we can replace the factor
\begin{align*}
\frac{(1-x^{r+1})}{x^r(1-x)(1-zx)}
=& (x^{-r} + x) \Big( \sum_{i \geq 0} x^i \Big) \Big( \sum_{j \geq 0} (zx)^j \Big) \\
=& (x^{-r} + x) \sum_{m \geq 0} (1 + z + \dots + z^m) x^m \\
=& \sum_{m \geq 0} (1 + z + \dots + z^m) x^{m-r} - \sum_{m \geq 0} (1 + z + \dots + z^m) x^{m+1}
\end{align*}
on the right-hand side \Cref{eq:long-proof-aux1} with
\begin{align*}
& x^{-r} + \sum_{j \geq 1} (1 + z + \dots + z^{rj}) x^{rj-r} - \sum_{j \geq 1} (1 + z + \dots + z^{rj-1}) x^{rj} \\
=& x^{-r} + \sum_{j \geq 0} (z^{rj} + z^{rj+1} + \dots + z^{r(j+1)})x^{rj} \\
=& x^{-r} + \sum_{j \geq 0} (1 + z^{1} + \dots + z^{r})z^{rj}x^{rj} \\
=& x^{-r} + \dfrac{1-z^{r+1}}{1-z} \cdot \dfrac{1}{1-z^rx^r}.
\end{align*}
The contribution of the term $x^{-r}$ above to the right-hand side of \Cref{eq:long-proof-aux1} is
\[
\sum_{k \geq 0} s^k [x^{kr}] \frac{(1-x^r) u}{(1-u)} x^{-r} = \dfrac{u}{1-u} \sum_{k \geq 0} s^k [x^{kr}] (x^{-r} - 1) = - \dfrac{u}{1-u}.
\]
Thus,
\begin{align*}
\sum_{n \geq 1} \sum_{k \geq 0} L_{\Delta'_{C_n,k+1}}(r) s^k u^n
=& \sum_{k \geq 0} s^k [x^{kr}] \frac{(1-x^r) (1-z^{r+1})u}
{(1-z) (1-u) (1-z^rx^r)} \\
=& \frac{(1-s) \left(1 - \left(\tfrac{1-us^2}{1-u}\right)^{r+1} \right) u}
{\left(1 - \left(\tfrac{1-us^2}{1-u}\right)\right) (1-u) \left(1 - \left(\tfrac{1-us^2}{1-u}\right)^r s \right)} \\
=& \frac{(1-s) \left( (1-u)^{r+1} - (1-us^2)^{r+1} \right) u}
{\left( (1-u)-(1-us^2) \right) (1-u) \left( (1-u)^r - (1-us^2)^r s \right)} \\
=& \frac{(1-s) ((1-u)^{r+1} - (1-us^2)^{r+1})}
{(s^2-1) \left( (1-u)^{r+1} - s (1-u) (1-us^2)^r \right)} \\
=& \frac{ (1-us^2)^{r+1} - (1-u)^{r+1}}
{(1+s) \left( (1-u)^{r+1} - s (1-u) (1-us^2)^r \right)} \qedhere
\end{align*}
\end{proof}

\section{Combinatorial Model: Poset on Signed Permutations}
\label{s:poset}

Using a novel statistic on signed permutations, we develop a combinatorial model for the poset $\mathcal{I}_{C_n}$ of alcoves in the fundamental parallelepiped introduced in \Cref{eq:defI}. We then use this model to prove \cref{t:cdes-basc}.

\subsection{Big ascents}

Endow $\dbra{n}$ with the order induced from $\ZZ$:
\[
\ol{n} < \dots < \ol{2} < \ol{1} < 1 < 2 < \dots < n.
\]
We write $\defnn{i \ll j}$ to indicate that $i < k < j$ for some $k \in \dbra{n}$. For instance, $4 \ll 8$ and $\ol{2} \ll 1$, but $\ol{1} \not\ll 1$. We use $\defnn{i^+}$ to denote the \defn{cyclic successor} of $i \in \dbra{n}$ in the order above. For example, $\ol{1}^+ = 1$, $2^+ = 3$, and $n^+ = \ol{n}$.

\begin{definition}\label{d:BAsc}
We say that a signed permutation $w \in \hypoct_n$ has a \defn{big ascent} at position $i \in \{\ol{1}\} \cup [n]$ if $w_i \ll w_{i^+}$.
Let $\defnn{\BAsc(w)}$ denote the set of big ascents of $w$ and $\defnn{\basc(w)} := |\BAsc(w)|$. Thus,
\begin{align*}
\ol{1} \in \BAsc(w) & \quad\text{if and only if}\quad w_1 \geq 2,\\
\text{for } i \in [n-1], \, i \in \BAsc(w) & \quad\text{if and only if}\quad w_i \ll w_{i+1},\\
n \in \BAsc(w) & \quad\text{if and only if}\quad w_n \leq \ol{2}.
\end{align*}
See \cref{f:PosetSign} for many examples of big ascents.
\end{definition}

Observe that $\BAsc(w)$ is always a proper subset of $\{\ol{1}\} \cup [n]$; otherwise, the inequalities
\[
\ol{w_1} \ll w_1 \ll \dots \ll w_n \ll \ol{w_n},
\]
would imply the existence of a strict chain of length $(n+2)+(n+1)$ within $\dbra{n}$. We invite the reader to verify that, in fact, $0 \leq \basc(w) \leq n-1$ for all signed permutations $w \in \hypoct_n$, and that the minimum value is attained by exactly two signed permutations.

\begin{remark}
When all the entries in the window notation of $w \in \hypoct_n$ are positive, the statistic $\basc(w)$ is precisely the big ascent statistic for unsigned permutations that appears in the work of Sanyal and Stump \cite{SS18Lipschitz}.
\end{remark}

\subsection{Poset structure}

Recall from \cref{def:setXn} that $X_n$ denotes the collection of signed permutations $w \in \hypoct_n$ for which $w^{-1}(1) > 0$.

\begin{definition}
\label{def:cover-rel}

Given $u,w\in X_n$ with $w = w_1 w_2 \dots w_{n-1} w_n$,
we write $\defnn{u \rightharpoonup w}$ if one of the following three conditions holds:
\begin{equation}\label{eq:cover-rel}
\left\lbrace
\begin{array}{lll}
w_1 \geq 2 & \text{and} &
u = \red{\ol{w_1}} w_2 \dots w_{n-1} w_n; \text{ or} \\
i \in \cap [n-1],\, w_i \ll w_{i+1} & \text{and} &
u = w_1 \dots w_{i-1} \red{w_{i+1} w_i} w_{i+2} \dots w_n; \text{ or} \\
w_n \leq \ol{2} & \text{and} &
u = w_1 w_2 \dots w_{n-1} \red{\ol{w_n}}.
\end{array}
\right.
\end{equation}
Thus, for each $w \in X_n$, there are exactly $\basc(w)$ signed permutations $u \in X_n$ such that $u \rightharpoonup w$.
\end{definition}

\begin{example}
Below we display three instances of the relation $\rightharpoonup$ in $X_4$.
\begin{align*}
\highlight{ \ol{3} } 1 \, \ol{4} \, 2 \rightharpoonup \highlight{ 3 } 1 \, \ol{4} \, 2,
\hspace*{.15\linewidth}
\ol{3} \highlight{ 1 \, \ol{4} } 2 \rightharpoonup \ol{3} \highlight{ \ol{4} \, 1 } 2,
\hspace*{.15\linewidth}
\ol{3} \, \ol{4}\, 1 \highlight{ 2 } \rightharpoonup \ol{3} \, \ol{4}\, 1 \highlight{ \ol{2} }.
\end{align*}
The permutations on the right have big ascents at positions $\ol{1}$, $2$, and $4$, respectively.
\end{example}

\begin{theorem}
\label{t:poset}
The relation $\rightharpoonup$
above
is the covering relation of a partial order $\leq$ on $X_n$.
Moreover, there is an isomorphism $A: X_n \to \mathcal{I}_{C_n}$ such that, for all $w \in X_n$, $A(w)$ lies inside the hypersimplex $\Delta_{C_n,\cdes(w^{-1})}$.
\end{theorem}

We prove \cref{t:poset} in \cref{ss:isom}. In the following section, we use this result to prove \cref{t:cdes-basc}.

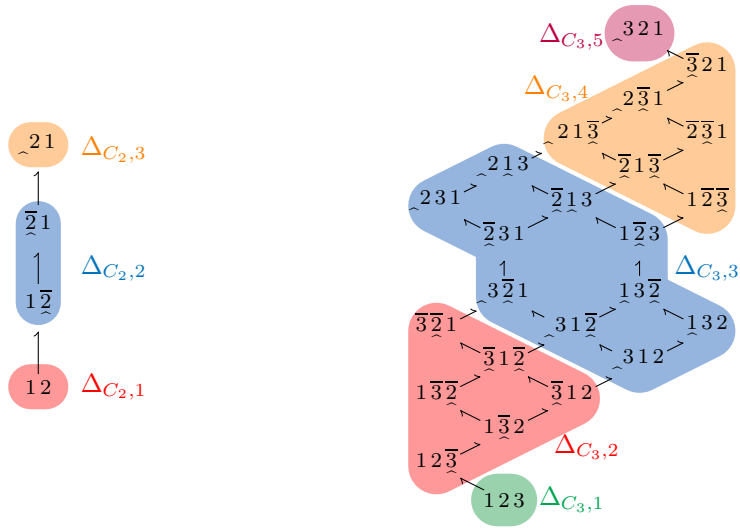
\begin{figure}[ht]
\[
\begin{gathered}\begin{tikzpicture}[xscale = 1, yscale = .7]
\node[inner sep = 2pt] (A) at (0, 0) {\scriptsize $1 \, 2$};
\node[red] at (1,0) {\small $\Delta_{C_2,1}$};
\node[inner sep = 2pt] (B) at (0,1.5) {\scriptsize $1 \, \ds{\ol{2}}$};
\node[inner sep = 2pt] (C) at (0, 3) {\scriptsize $\ds{\ol{2}} \, 1$};
\node[RoyalBlue] at (1,2.25) {\small $\Delta_{C_2,2}$};
\node[inner sep = 2pt] (D) at (0,4.5) {\scriptsize $\ds{\,}\, 2\, 1$};
\node[orange] at (1,4.5) {\small $\Delta_{C_2,3}$};
\draw[-left to] (A) -- (B);
\draw[-left to] (B) -- (C);
\draw[-left to] (C) -- (D);
\begin{scope}[on background layer , shift = {(0,.1)}]
\path[fill=red!40,draw=red!40,line width=.6cm, line cap=round, line join=round,shift={(0,-.1)}]
(-.1,0) -- (.1,0);
\path[fill=RoyalBlue!40,draw=RoyalBlue!40,line width=.6cm, line cap=round, line join=round]
(0,1.5) -- (0,3);
\path[fill=orange!40,draw=orange!40,line width=.6cm, line cap=round, line join=round]
(-.1,4.5) -- (.1,4.5);
\end{scope}
\end{tikzpicture}\end{gathered}
\hspace*{.2\linewidth}
\begin{gathered}\begin{tikzpicture}[xscale = .9, yscale = .45]
\node[inner sep = 1pt] (123) at ( 0, 0) {\scriptsize $1\, 2\, 3$};
\node[Green] at (1,0) {\small $\Delta_{C_3,1}$};
\node[inner sep = 1pt] (n123) at (-1, 1) {\scriptsize $1\, 2\, \ds{\ol{3}}$};
\node[inner sep = 1pt] (2n13) at ( 0, 2) {\scriptsize $1\, \ds{\ol{3}}\, 2$};
\node[red,shift={(-.7,-.2)}] at (2,2) {\small $\Delta_{C_3,2}$};
\node[inner sep = 1pt] (n2n13) at (-1, 3) {\scriptsize $1\, \ol{3}\, \ds{\ol{2}}$};
\node[inner sep = 1pt] (23n1) at ( 1, 3) {\scriptsize $\ds{\ol{3}}\, 1\, 2$};
\node[inner sep = 1pt] (n23n1) at ( 0, 4) {\scriptsize $\ds{\ol{3}}\, 1\, \ds{\ol{2}}$};
\node[inner sep = 1pt] (231) at ( 2, 4) {\scriptsize $\ds{\,}\,3\, 1\, 2$};
\node[RoyalBlue] at (3,6.85) {\small $\Delta_{C_3,3}$};
\node[inner sep = 1pt] (3n2n1) at (-1, 5) {\scriptsize $\ol{3}\, \ds{\ol{2}}\, 1$};
\node[inner sep = 1pt] (n231) at ( 1, 5) {\scriptsize $\ds{\,}\,3\, 1\, \ds{\ol{2}}$};
\node[inner sep = 1pt] (213) at ( 3, 5) {\scriptsize $\ds{1}\, 3\, 2$};
\node[inner sep = 1pt] (3n21) at ( 0, 6) {\scriptsize $\ds{\,}\, 3\, \ds{\ol{2}}\, 1$};
\node[inner sep = 1pt] (n213) at ( 2, 6) {\scriptsize $\ds{1}\, 3\, \ds{\ol{2}}$};
\node[inner sep = 1pt] (31n2) at ( 0, 7.7) {\scriptsize $\ds{\ol{2}}\, 3\, 1$};
\node[inner sep = 1pt] (1n23) at ( 2, 7.7) {\scriptsize $1\, \ds{\ol{2}}\, 3$};
\node[inner sep = 1pt] (312) at (-1, 8.7) {\scriptsize $\ds{\,}\,2\, 3\, 1$};
\node[inner sep = 1pt] (13n2) at ( 1, 8.7) {\scriptsize $\ds{\ol{2}}\, \ds{1}\, 3$};
\node[inner sep = 1pt] (n1n23) at ( 3, 8.7) {\scriptsize $1\, \ol{2}\, \ds{\ol{3}}$};
\node[orange,shift={(.7,.2)}] at (0,11.7) {\small $\Delta_{C_3,4}$};
\node[inner sep = 1pt] (132) at ( 0, 9.7) {\scriptsize $\ds{\,}\, 2\, \ds{1}\, 3$};
\node[inner sep = 1pt] (n13n2) at ( 2, 9.7) {\scriptsize $\ds{\ol{2}}\, 1\, \ds{\ol{3}}$};
\node[inner sep = 1pt] (n132) at ( 1,10.7) {\scriptsize $\ds{\;}\;2\, 1\, \ds{\ol{3}}$};
\node[inner sep = 1pt] (3n1n2) at ( 3,10.7) {\scriptsize $\ol{2}\, \ds{\ol{3}}\, 1$};
\node[inner sep = 1pt] (3n12) at ( 2,11.7) {\scriptsize $\ds{\,}\, 2\, \ds{\ol{3}}\, 1$};
\node[inner sep = 1pt] (32n1) at ( 3,12.7) {\scriptsize $\ds{\ol{3}}\, 2\, 1$};
\node[inner sep = 1pt] (321) at ( 2,13.7) {\scriptsize $\ds{\,}\,3\, 2\, 1$};
\node[purple] at (1,13.7) {\small $\Delta_{C_3,5}$};

\draw[-left to] (123) -- (n123);
\draw[-left to] (n123) -- (2n13);
\draw[-left to] (2n13) -- (n2n13);
\draw[-left to] (2n13) -- (23n1);
\draw[-left to] (n2n13) -- (n23n1);
\draw[-left to] (23n1) -- (n23n1);
\draw[-left to] (23n1) -- (231);
\draw[-left to] (n23n1) -- (3n2n1);
\draw[-left to] (n23n1) -- (n231);
\draw[-left to] (231) -- (n231);
\draw[-left to] (231) -- (213);
\draw[-left to] (3n2n1) -- (3n21);
\draw[-left to] (213) -- (n213);
\draw[-left to] (n231) -- (3n21);
\draw[-left to] (n231) -- (n213);
\draw[-left to] (3n21) -- (31n2);
\draw[-left to] (n213) -- (1n23);
\draw[-left to] (31n2) -- (312);
\draw[-left to] (31n2) -- (13n2);
\draw[-left to] (1n23) -- (13n2);
\draw[-left to] (1n23) -- (n1n23);
\draw[-left to] (312) -- (132);
\draw[-left to] (13n2) -- (132);
\draw[-left to] (13n2) -- (n13n2);
\draw[-left to] (n1n23) -- (n13n2);
\draw[-left to] (132) -- (n132);
\draw[-left to] (n13n2) -- (n132);
\draw[-left to] (n13n2) -- (3n1n2);
\draw[-left to] (n132) -- (3n12);
\draw[-left to] (3n1n2) -- (3n12);
\draw[-left to] (3n12) -- (32n1);
\draw[-left to] (32n1) -- (321);
\begin{scope}[on background layer]
\path[fill=Green!40,draw=Green!40,line width=.7cm, line cap=round, line join=round]
(-.1,0) -- (.1,0);
\path[fill=red!40,draw=red!40,line width=.75cm, line cap=round, line join=round]
(-1,1) -- (1,3) -- (-1,5) -- cycle;
\path[fill=RoyalBlue!40,draw=RoyalBlue!40,line width=.75cm, line cap=round, line join=round]
(2,4) -- (0,6) -- (0,7.7) -- (-1,8.7) -- (0,9.7) -- (2,7.7) -- (2,6) -- (3,5) -- cycle ;
\path[fill=orange!40,draw=orange!40,line width=.75cm, line cap=round, line join=round]
(3,12.7) -- (1,10.7) -- (3,8.7) -- cycle;
\path[fill=purple!40,draw=purple!40,line width=.75cm, line cap=round, line join=round]
(1.9,13.8) -- (2.1,13.8);
\end{scope}
\end{tikzpicture}\end{gathered}
\]
\caption{
The poset $(X_n , \leq)$ for $n=2$ (left) and $n=3$ (right) with the big ascent positions $i \in \{\ol{1}\} \cup [n]$
underlined with a little wedge ``$\wh{\phantom{a}}$''.
Elements $w \in X_n$ are grouped by the value of $\cdes(w^{-1})$.
}
\label{f:PosetSign}
\end{figure}

\subsection{Proof of Theorem \ref{t:cdes-basc}}
\label{ss:proof2}

Recall that we want to prove that for all $n \geq 1$ and $k \geq 1$,
\[
h_{\Delta'_{C_n,k}}^*(t) = \sum_{ \substack{ w \in X_n \\ \cdes(w^{-1}) = k } } t^{\basc(w)}.
\]
To achieve this, we use the following result of Stanley regarding the $h^*$-polynomial of polytopes with a \emph{shellable unimodular triangulation}.

Let $P$ be a polytope with a unimodular triangulation $\Gamma$. Suppose there exists a linear ordering\footnote{Such an ordering is called a \defn{shelling} of $\Gamma$, and it does not necessarily exist.} $T_1 \prec T_2 \prec \dots $ of the maximal simplices in $\Gamma$ such that the intersection of $T_k$ with $T_1 \cup \dots \cup T_{k-1}$ is a union of facets of $T_k$. Let $c(T_k)$ be the number of facets of $T_k$ in this intersection. Then, a result of Stanley \cite{S80decomp} shows that
\begin{equation}\label{eq:Stanley-shelling}
h^*_P(t) = \sum_{i} t^{c(T_i)}.
\end{equation}

\begin{proof}[{Proof of \cref{t:cdes-basc}}]
For $n \geq 1$ and $1 \leq k \leq 2n-1$, let $P_{n,k} = \{ x \in \Pi_{C_n} \mid 2 x_1 \leq k \}$.
The polytope $P_{n,k}$ is the disjoint union of $\Delta'_{C_n,1},\dots,\Delta'_{C_n,k}$, and so
$L_{P_{n,k}}(r) = L_{\Delta'_{C_n,1}}(r) + \dots + L_{\Delta'_{C_n,k}}(r)$.
Moreover, since all these polytopes have the same dimension,
\[h^*_{P_{n,k}}(t) = h^*_{\Delta'_{C_n,1}}(t) + \dots + h^*_{\Delta'_{C_n,k}}(t).\]

The collection $\mathcal{P}_{n,k}$ of alcoves contained in $P_{n,k}$ forms an order ideal of $\mathcal{I}_{C_n}$, since the hyperplane $H_{\theta,k}$ separates the fundamental alcove $A_\circ$ from every alcove of $\mathcal{I}_{C_n}$ that is not in $\mathcal{P}_{n,k}$, and no others.
Fix any linear extension $\prec$ of $\mathcal{I}_{C_n}$ such that alcoves in $\mathcal{P}_{n,k}$ appear before those in $\mathcal{P}_{n,k+1} \setminus \mathcal{P}_{n,k}$.
Since every linear extension of the weak order on a hyperplane arrangement
\footnote{We can restrict to a finite linear arrangement by forgetting
the hyperplanes not intersecting $\Pi_{C_n}$ and taking the cone over them.}
gives a shelling of the complex of regions (see for instance \cite[Proposition 3.4]{reading05fans}),
$\prec$ defines a shelling on the alcove triangulation of each $P_{n,k}$.
Moreover, the number of facets $c(A)$ of an alcove $A$ that appear in the intersection with the preceding alcoves is precisely the number of elements covered by $A$ in the weak order.
Therefore, by \cref{t:poset} and Stanley's formula \cref{eq:Stanley-shelling}, we have
\begin{equation}\label{eq:h-star-Pnk}
h^*_{P_{n,k}}(t) =
\sum_{ \substack{ w \in X_n : \\ \cdes(w^{-1}) \leq k } } t^{\basc(w)},
\end{equation}
and therefore
\[
h^*_{\Delta'_{C_n,k}}(t) = h^*_{P_{n,k}}(t) - h^*_{P_{n,k-1}}(t) = \sum_{ \substack{ w \in X_n : \\ \cdes(w^{-1}) = k } } t^{\basc(w)}. \qedhere
\]
\end{proof}

Taking $k = 2n-1$ in \Cref{eq:h-star-Pnk}, we obtain the following result.

\begin{corollary}
\label{cor:h*-parallel}
The $h^*$-polynomial of the type C fundamental parallelepiped is
\[
h^*_{\Pi_{C_n}}(t) = \sum_{ w \in X_n } t^{\basc(w)} = \Psi_{C_n}(t),
\]
where $\Psi_{C_n}(t)$ is the polynomial defined in \Cref{eq:def-Psi}.
\end{corollary}

\begin{example}
We can read from \cref{f:PosetSign} that
\[
h^*_{\Delta'_{C_2,1}}(t) = 1,\qquad
h^*_{\Delta'_{C_2,2}}(t) = 2t,\qquad
h^*_{\Delta'_{C_2,3}}(t) = t; \qquad \text{and}
\]
\[
h^*_{\Delta'_{C_3,1}}(t) = 1, \quad
h^*_{\Delta'_{C_3,2}}(t) = 5t + t^2, \quad
h^*_{\Delta'_{C_3,3}}(t) = 5t + 5t^2,\quad
h^*_{\Delta'_{C_3,4}}(t) = 3t + 3t^2,\quad
h^*_{\Delta'_{C_3,5}}(t) = t.
\]
Consequently,
\[
h^*_{\Pi_{C_2}}(t) = 1 + 3t,
\qquad\text{and}\qquad
h^*_{\Pi_{C_3}}(t) = 1 + 14t + 9t^2.
\]
More values of $h^*_{\Delta'_{C_n,k}}(t)$ and $h^*_{\Pi_{C_n}}(t)$ can be found in \cref{tab:h-half-open}.
\end{example}

\afterpage{
\clearpage
\newgeometry{left=1.5in,right=1.5in,bottom=.1in,top=.7in}
\begin{landscape}

\begin{table}
{
\centering
\scriptsize

\vfill

\begin{tabular}{c|lllllll}
$k \backslash n$ & $1$ & $2$ & $3$ & $4$ & $5$ & $6$ & $7$ \\
\hline
$1$ & $1$ & $1$ & $1$ & $1$ & $1$ & $1$ & $1$ \\
$2$ & & $2t$ & $5t + t^2$ & $9t + 5t^2$ & $14t + 15t^2 + t^3$ & $20t + 35t^2 + 7t^3$ & $27t + 70t^2 + 28t^3 + t^4$ \\
$3$ & & $t$ & $5t + 5t^2$ & $13t + 30t^2 + 4t^3$ & $26t + 106t^2 + 43t^3 + t^4$ & $45t + 287t^2 + 238t^3 + 27t^4$ & $71t + 658t^2 + 932t^3 + 257t^4 + 8t^5$ \\
$4$ & & & $3t + 3t^2$ & $13t + 42t^2 + 13t^3$ & $35t + 223t^2 + 178t^3 + 14t^4$ & $75t + 796t^2 + 1211t^3 + 304t^4 + 6t^5$ & $140t + 2261t^2 + 5689t^3 + 3013t^4 + 278t^5 + t^6$ \\
$5$ & & & $t$ & $9t + 29t^2 + 9t^3$ & $35t + 268t^2 + 268t^3 + 35t^4$ & $96t + 1334t^2 + 2681t^3 + 1006t^4 + 45t^5$ & $216t + 4838t^2 + 16481t^3 + 12485t^4 + 1920t^5 + 27t^6$ \\
$6$ & & & & $4t + 9t^2 + t^3$ & $26t + 199t^2 + 199t^3 + 26t^4$ & $96t + 1505t^2 + 3410t^3 + 1505t^4 + 96t^5$ & $267t + 7245t^2 + 29221t^3 + 26813t^4 + 5446t^5 + 140t^6$ \\
$7$ & & & & $t$ & $14t + 88t^2 + 69t^3 + 5t^4$ & $75t + 1175t^2 + 2662t^3 + 1175t^4 + 75t^5$ & $267t + 7930t^2 + 34549t^3 + 34549t^4 + 7930t^5 + 267t^6$ \\
$8$ & & & & & $5t + 19t^2 + 6t^3$ & $45t + 622t^2 + 1243t^3 + 462t^4 + 20t^5$ & $216t + 6413t^2 + 27937t^3 + 27937t^4 + 6413t^5 + 216t^6$ \\
$9$ & & & & & $t$ & $20t + 206t^2 + 301t^3 + 69t^4 + t^5$ & $140t + 3786t^2 + 15229t^3 + 13928t^4 + 2813t^5 + 71t^6$ \\
$10$ & & & & & & $6t + 34t^2 + 21t^3 + t^4$ & $71t + 1568t^2 + 5261t^3 + 3900t^4 + 575t^5 + 7t^6$ \\
$11$ & & & & & & $t$ & $27t + 415t^2 + 981t^3 + 468t^4 + 35t^5$ \\
$12$ & & & & & & & $7t + 55t^2 + 56t^3 + 8t^4$ \\
$13$ & & & & & & & $t$ \\
\hline
Sum & $1$ & $1 + 3t$ & $1 + 14t + 9t^2$ & $1 + 49t + 115t^2 + 27t^3$ & $1 + 156t + 918t^2 + 764t^3 + 81t^4$ & $1 + 479t + 5994t^2 + 11774t^3 + 4549t^4 + 243t^5$ & $1 + 1450t + 35239t^2 + 136364t^3 + 123359t^4 + 25418t^5 + 729t^6$
\end{tabular}
\caption{$h^*$-polynomial of the half-open hypersimplcies $\Delta'_{C_n,k}$ and of the type C fundamental parallelepiped $\Pi_{C_n}$ (last row).}
\label{tab:h-half-open}

\vspace*{.95in}

\begin{tabular}{c|lllllll}
$k \backslash n$ & $1$ & $2$ & $3$ & $4$ & $5$ & $6$ & $7$ \\
\hline
$1$ & $1$ & $1$ & $1$ & $1$ & $1$ & $1$ & $1$ \\
$2$ & & $1 + t$ & $1 + 4t + t^2$ & $1 + 8t + 5t^2$ & $1 + 13t + 15t^2 + t^3$ & $1 + 19t + 35t^2 + 7t^3$ & $1 + 26t + 70t^2 + 28t^3 + t^4$ \\
$3$ & & $1$ & $1 + 6t + 3t^2$ & $1 + 17t + 26t^2 + 3t^3$ & $1 + 34t + 102t^2 + 38t^3 + t^4$ & $1 + 58t + 288t^2 + 224t^3 + 26t^4$ & $1 + 90t + 673t^2 + 904t^3 + 250t^4 + 8t^5$ \\
$4$ & & & $1 + 4t + t^2$ & $1 + 21t + 39t^2 + 7t^3$ & $1 + 55t + 237t^2 + 147t^3 + 10t^4$ & $1 + 113t + 878t^2 + 1134t^3 + 261t^4 + 5t^5$ & $1 + 203t + 2519t^2 + 5612t^3 + 2795t^4 + 251t^5 + t^6$ \\
$5$ & & & $1$ & $1 + 17t + 26t^2 + 3t^3$ & $1 + 64t + 306t^2 + 216t^3 + 19t^4$ & $1 + 164t + 1590t^2 + 2572t^3 + 805t^4 + 30t^5$ & $1 + 348t + 5789t^2 + 16832t^3 + 11380t^4 + 1596t^5 + 21t^6$ \\
$6$ & & & & $1 + 8t + 5t^2$ & $1 + 55t + 237t^2 + 147t^3 + 10t^4$ & $1 + 185t + 1920t^2 + 3320t^3 + 1135t^4 + 51t^5$ & $1 + 475t + 9248t^2 + 30824t^3 + 24265t^4 + 4229t^5 + 90t^6$ \\
$7$ & & & & $1$ & $1 + 34t + 102t^2 + 38t^3 + t^4$ & $1 + 164t + 1590t^2 + 2572t^3 + 805t^4 + 30t^5$ & $1 + 526t + 10765t^2 + 37436t^3 + 30877t^4 + 5746t^5 + 141t^6$ \\
$8$ & & & & & $1 + 13t + 15t^2 + t^3$ & $1 + 113t + 878t^2 + 1134t^3 + 261t^4 + 5t^5$ & $1 + 475t + 9248t^2 + 30824t^3 + 24265t^4 + 4229t^5 + 90t^6$ \\
$9$ & & & & & $1$ & $1 + 58t + 288t^2 + 224t^3 + 26t^4$ & $1 + 348t + 5789t^2 + 16832t^3 + 11380t^4 + 1596t^5 + 21t^6$ \\
$10$ & & & & & & $1 + 19t + 35t^2 + 7t^3$ & $1 + 203t + 2519t^2 + 5612t^3 + 2795t^4 + 251t^5 + t^6$ \\
$11$ & & & & & & $1$ & $1 + 90t + 673t^2 + 904t^3 + 250t^4 + 8t^5$ \\
$12$ & & & & & & & $1 + 26t + 70t^2 + 28t^3 + t^4$ \\
$13$ & & & & & & & $1$ \\
\end{tabular}
\caption{$h^*$-polynomial of the (closed) hypersimplcies $\Delta_{C_n,k}$.}
\label{tab:h-closed}

\vfill
}
\end{table}

\end{landscape}
\aftergroup
\restoregeometry
\clearpage
\restoregeometry
}

\subsection{Type C circular descents (revisited)}

In this section, we refine the circular descent statistic of Lam and Postnikov (\cref{def:LP-cdes}) according to our choice of simple roots for the root system of type $C_n$.

\begin{definition}
We say that a signed permutation $w \in \hypoct_n$ has a \defn{circular descent} at position $i \in \dbra{n}$ if $w_i > w_{i^+}$. Let $\defnn{\CDes(w)}$ denote the set of circular descents of $w$.
\end{definition}

Circular descents are easy to read using the complete notation of $w \in \hypoct_n$. Namely, a position $i \in \dbra{n}$ is a circular descent of $w$ if, in its complete notation $\wt{w}$, the letter $w_{i}$ is (cyclically) followed by a smaller letter.

\begin{example}\label{ex:cdes}
Consider the signed permutation $u \in \hypoct_n$ with complete notation
\[
\wt{u} = \ol{5} \, \ol{1} \, \underset{\cdessymb}{4} \, \underset{\cdessymb}{2} \, \ol{3} \, \underset{\cdessymb}{3} \, \underset{\cdessymb}{\ol{2}} \, \ol{4} \, 1 \, \underset{\cdessymb}{5}.
\]
The marked positions show its circular descents.
That is, $\CDes(u) = \{\ol{3},\ol{2},1,2,5 \}$ and $\cdes(u) = 5$.
\end{example}

The next result follows by a direct comparison between the definition of the statistic $\cdes$ and of the set $\CDes$.
For completeness, we include its proof.

\begin{lemma}
For all $ w \in \hypoct_n$, \[ \cdes(w) = |\CDes(w)|. \]
\end{lemma}

\begin{proof}
Recall that
$
\cdes(w) = d_0(w) + d_1(w) + 2 d_2(w) + \dots + 2 d_n(w),
$
where
\begin{itemize}
\item $d_0(w)=1 \Leftrightarrow
w(- 2 e_n) \in \Phi^- \Leftrightarrow
w_n > 0 \Leftrightarrow
w_n > w_{\ol{n}} \Leftrightarrow
n \in \CDes(w)$;
\item $d_1(w)=1 \Leftrightarrow
w(2 e_1) \in \Phi^- \Leftrightarrow
w_1 < 0 \Leftrightarrow
w_1 < w_{\ol{1}} \Leftrightarrow
\ol{1} \in \CDes(w)$;
\item for $2 \leq i \leq n$,
$d_i(w) = 1 \Leftrightarrow
w(e_i - e_{i-1}) \in \Phi^- \Leftrightarrow
w_{i-1} > w_i \Leftrightarrow
i-1 \in \CDes(w)$.
\end{itemize}
Since for $i \in [n-1]$, $i \in \CDes(w)$ if and only if $\ol{i+1} \in \CDes(w)$, this proves the claimed equality.
\end{proof}

\begin{remark}
The statistic $\cdes$ depends on the choice of simple system for $\Phi$.
We warn the reader that for the other usual choice of simple roots for the type C root system, the description of $\cdes$ in \cite[Section 12]{LP18alcoved2} contains a minor typo.
Indeed, for $i \in [n-1]$, $w$ has a descent (in the Lam-Postnikov convention) if either $w_i > w_{i+1}$ and both have the same sign, or if $w_i < 0 < w_{i+1}$.
\end{remark}

In what follows, it will be important to have an explicit description of the circular descents of the inverse of a signed permutation $w$. The complete notation of $w$ will again play an important role.

\begin{lemma}\label{l:alt_CDes}
For $w \in \hypoct_n$ and $i \in \dbra{n}$,
\[
i \in \CDes(w^{-1})
\quad\text{if and only if}\quad
i^+ \text{ precedes } i \text{ in } \wt{w}.
\]
\end{lemma}

Observe that $\ol{1}$ is a circular descent of $w^{-1}$ if and only if $w^{-1}(1) < 0$.
Therefore, $X_n$ is exactly the set of signed permutations $w$ for which $\ol{1} \notin \CDes(w^{-1})$.

\begin{example}\label{ex:cdes-inverse}
Consider the signed permutation $w = 4 \, \ol{2} \, 1 \, \ol{3} \, 5 \in X_5$. It is the inverse of the permutation in \cref{ex:cdes}, so $\CDes(w^{-1}) = \{\ol{3},\ol{2},1,2,5 \}$. The complete notation of $w$ is
\[
\wt{w} = \ol{5} \, 3 \, \ol{1} \, 2 \, \ol{4} \, 4 \, \ol{2} \, 1 \, \ol{3} \, 5.
\]
We see, for instance, that $2 \in \CDes(w^{-1})$ because $3 = 2^+$ precedes $2$ in $\wt{w}$.
Similarly, $\ol{2} \in \CDes(w^{-1})$ since $\ol{1} = \ol{2}^+$ appears before $\ol{2}$;
and $5 \in \CDes(w^{-1})$ since $\ol{5} = 5^+$ appears before $5$ in $\wt{w}$.
\end{example}

\subsection{Poset isomorphism}
\label{ss:isom}

In this section we construct an isomorphism $A: X_n \to \mathcal{I}_{C_n}$ by explicitly constructing the vertices of the alcove $A(w)$ for each signed permutation $w \in X_n$; thereby proving \cref{t:poset}.

Given $w \in X_n$, let $v^{n+1}(w) \in \ZZ^n \subset \RR^n$ be the point with coordinates
\begin{equation}\label{eq:def-trans-vector}
\defnn{v^{n+1}_i(w)} :=
\left| [i-1] \cap \CDes(w^{-1}) \right|, \quad\text{ for } i = 1,2,\dots,n.
\end{equation}
Thus, $v^{n+1}_1(w) = 0$. For $k \in [n]$, let $v^k(w) \in \tfrac{1}{2} \ZZ^n$ be the point
\[
\defnn{v^k(w)} := v^{k+1}(w) + \tfrac{1}{2} e_{w_k} = v^{n+1}(w) + \tfrac{1}{2}\big( e_{w_k} + \dots + e_{w_n} \big).
\]

\begin{example}\label{ex:per-to-simp-I}
Let $w = 4 \, \ol{2} \, 1 \, \ol{3} \, 5 \in X_5$ be the signed permutation from \cref{ex:cdes-inverse}.
Then,
\[
[n-1] \cap \CDes(w^{-1}) = [4] \cap \CDes(w^{-1}) = \{1, 2\},
\]
and
\begin{align*}
v^6(w) & = (0, 1, 2, 2, 2 );\\
v^5(w) = v^6(w) + \tfrac{1}{2}e_5 & = (0, 1, 2, 2, \tfrac{5}{2});\\
v^4(w) = v^5(w) + \tfrac{1}{2}e_{\ol{3}} & = (0, 1, \tfrac{3}{2}, 2, \tfrac{5}{2});\\
v^3(w) = v^4(w) + \tfrac{1}{2}e_1 & = (\tfrac{1}{2}, 1, \tfrac{3}{2}, 2, \tfrac{5}{2});\\
v^2(w) = v^3(w) + \tfrac{1}{2}e_{\ol{2}} & = (\tfrac{1}{2}, \tfrac{1}{2}, \tfrac{3}{2}, 2, \tfrac{5}{2});\\
v^1(w) = v^2(w) + \tfrac{1}{2}e_4 & = (\tfrac{1}{2}, \tfrac{1}{2}, \tfrac{3}{2}, \tfrac{5}{2}, \tfrac{5}{2}).
\end{align*}
\end{example}

Since $\{e_{w_1} , \dots , e_{w_n}\}$ is a linear basis of $\RR^n$, the vectors $v^1(w),\dots,v^{n+1}(w)$ are affinely independent and their Convex hull is a full-dimensional simplex.

\begin{definition}\label{d:map-A}
Given $w \in X_n$, let $\defnn{A(w)}$ be the simplex
\[
\defnn{A(w)} := \Conv\{ v^k(w) \mid k = 1,\dots ,n+1 \}.
\]
\end{definition}

\begin{example}\label{ex:idAlc}
For the identity permutation ${\rm id}_n := 12\cdots n$, $\CDes({\rm id}_n^{-1})=\{n\}$ thus $v^{n+1}({\rm id}_n)$ is the zero vector and $v^k({\rm id}_n) = (0,\cdots , 0, \tfrac{1}{2},\cdots,\tfrac{1}{2})$ has $n-k+1$ nonzero coordinates. This is exactly the fundamental alcove; that is, $A({\rm id}_n) = A_\circ$.
\end{example}

\begin{proof}[Proof of \cref{t:poset}]
We prove that the map $w \to A(w)$ defines a poset isomorphism $A : X_n \to \mathcal{I}_{C_n}$.
To do so, we first show that $A$ is a well-defined bijection, and then we show that it is both order-preserving and order-reflecting.

The group $\wt{W}_{C_n}$ contains all the translations by integer vectors.
The simplex $A(w)$ is the translation of the alcove $w(A_\circ)$ by the integer vector $v^{n+1}(w)$ in \Cref{eq:def-trans-vector}, and is therefore an alcove of $\HH_{C_n}$.
To show that $A(w) \in \mathcal{I}_{C_n}$, in fact that $A(w)$ is contained in the hypersimplex $\Delta_{C_n,\cdes(w^{-1})}$, we verify that the vertices of $A(w)$ satisfy the inequalities in \Cref{eq:def-hypersimplex}.
Fix $w \in X_n$ and write $v^j = v^j(w)$ for all $j \in [n+1]$.
\begin{itemize}
\item
$2v_1^{n+1} = 2|\varnothing| = 0$ and, since $1$ appears in the window notation of $w$,
$v_1^j \in \{ v_1^{n+1} , v_1^{n+1} + \tfrac{1}{2} \}$ for all $j$.
Therefore,
\[
0 \leq 2v_1^j \leq 1 \text{ for all } j.
\]

\item
For $k = 2,\dots,n$, we have
$v^{n+1}_k - v^{n+1}_{k-1} = \begin{cases}
1 & \text{if } k-1 \in \CDes(w^{-1}),\\
0 & \text{otherwise.}
\end{cases}$\newpage
If $k-1 \in \CDes(w^{-1})$, then by \cref{l:alt_CDes} we have that
$k$ precedes $k-1$ and $\ol{k-1}$ precedes $\ol{k}$ in $\wt{w}$.
Thus, if $k$ (resp. $\ol{k-1}$) appears in the window notation of $w$, and therefore some $v^j$ is obtained from $v^{j+1}$ by adding $\tfrac{1}{2}e_k$ (resp. $-\tfrac{1}{2}e_{k-1}$), then $\tfrac{1}{2}e_{k-1}$ (resp. $-\tfrac{1}{2}e_k$) was added for an earlier vertex $v^{j'}$ with $j' > j$. Hence, $\tfrac{1}{2} \leq v^j_k - v^j_{k-1} \leq 1$ for all $j$.
A similar analysis shows that if $k-1 \notin \CDes(w^{-1})$, then $0 \leq v^j_k - v^j_{k-1} \leq \tfrac{1}{2}$ for all $j$.
In any case, we have
\[
0 \leq v^j_k - v^j_{k-1} \leq 1 \text{ for all } j.
\]

\item
$2 v^{n+1}_n = 2 | [n-1] \cap \CDes(w) | = \begin{cases}
\cdes(w^{-1}) - 1 & \text{if } n \in \CDes(w^{-1}),\\
\cdes(w^{-1}) & \text{otherwise.}
\end{cases}$\newline
If $n \in \CDes(w^{-1})$, then $n$ appears in the window notation of $w$ and therefore
$v_n^j \in \{ v_n^{n+1} , v_n^{n+1} + \tfrac{1}{2} \}$ for all $j$.
If, on the other hand, $n \notin \CDes(w^{-1})$, then $\ol{n}$ appears in the window notation of $w$ and therefore
$v_n^j \in \{ v_n^{n+1} , v_n^{n+1} - \tfrac{1}{2} \}$ for all $j$.
In both cases, we have
\[
\cdes(w^{-1}) - 1 \leq 2 v_n^j \leq \cdes(w^{-1}) \text{ for all } j.
\]
\end{itemize}
Thus $A(w) \subseteq \Delta_{C_n,\cdes(w^{-1})} \subseteq \Pi_{C_n}$.
Observe that $w$ can be recovered from $A(w)$ by first ordering its vertices according to their number of non-integer coordinates and then keeping track of what coordinate becomes increases/decreases from one vertex to the next.
Therefore, since both $X_n$ and $\mathcal{I}_{C_n}$ have the same cardinality (see \Cref{eq:volPi}), the map $A$ is a bijection.

We will now prove that $A(u) \lessdot A(w)$ if and only if $u \rightharpoonup w$.

First, suppose $u,w \in X_n$ are such that $A(u) \lessdot A(w)$.
Let $\alpha \in \Phi^+$ such that the common facet of $A(u)$ and $A(w)$ lies on a hyperplane of the form $H_{\alpha,p}$.
Given that the hyperplane $H_{\alpha,p}$ cuts the interior of $\Pi_{C_n}$, the root $\alpha$ is not simple.
Moreover, since the reflection of a lattice point across $H_{\alpha,p}$ does not change the number of non-integer coordinates, there is a $k \in [n+1]$ such that $v^j(u) = v^j(w)$ for $j \neq k$ and $v^k(w) = v^k(u) + \tfrac{1}{2} \alpha$.

\noindent\textbf{Case 1:}
If $k = 1$,
then $u_i = w_i$ for all $i > 1$ and necessarily $u_1 = \ol{w_1}$.
Moreover, since
\[
v^1(w) - v^1(u) = ( v^2(w) + \tfrac{1}{2}e_{w_1} ) - ( v^2(u) + \tfrac{1}{2}e_{u_1} ) = e_{w_1}
\]
is, up to scaling, a positive but not simple root, $w_1 \geq 2$.
Thus, $u \rightharpoonup w$ by the first case in \cref{def:cover-rel}.

\noindent\textbf{Case 2:}
If $k \in [2,n]$,
then $u_i = w_i$ for all $i \neq k-1,k$ and since
\[
\tfrac{1}{2}(e_{w_{k-1}} + e_{w_{k}}) = v^{k-1}(w) - v^{k+1}(w)
= v^{k-1}(u) - v^{k+1}(u)
= \tfrac{1}{2}(e_{u_{k-1}} + e_{u_{k}}),
\]
necessarily $u_{k-1} = w_k$ and $u_k = w_{k-1}$.
Moreover, since
\[
v^k(w) - v^k(u) = \tfrac{1}{2} (e_{w_k} - e_{u_k}) = \tfrac{1}{2} (e_{w_k} - e_{w_{k-1}})
\]
is, up to scaling, a positive but not simple root, $w_{k-1} \ll w_k$.
Thus, $u \rightharpoonup w$ by the second case in \cref{def:cover-rel}.

\noindent\textbf{Case 3:}
If $k = n+1$,
then $u_i = w_i$ for all $i < n$ and necessarily $u_n = \ol{w_n}$.
Moreover, since
\[
v^{n+1}(w) - v^{n+1}(u) = ( v^n(w) - \tfrac{1}{2}e_{w_n} ) - ( v^n(u) - \tfrac{1}{2}e_{u_n} ) = - e_{w_n}
\]
is, up to scaling, a positive but not simple root, $w_n \leq \ol{2}$.
Thus, $u \rightharpoonup w$ by the third case in \cref{def:cover-rel}.

Now we prove the reverse implication. Suppose $u,w \in X_n$ are such that $u \rightharpoonup w$.
We will show that $A(u)$ and $A(w)$ share $n$ vertices and that, for the non-common vertices, $v^k(w) - v^k(u)$ is in the direction of a positive root; thereby showing that $A(u) \lessdot A(w)$.

\noindent\textbf{Case 1:}
If $u = \ol{w_1} w_2 \dots w_n$ and $w_1 \in [2,n]$, then
\[
\CDes(w^{-1}) =
\begin{cases}
\CDes(w^{-1}) \cup \{ n \} & \text{if } \ol{u_1} = n,\\
\CDes(w^{-1}) & \text{if } \ol{u_1} \in [2,n-1].
\end{cases}
\]
In any case, $\CDes(w^{-1}) \cap [n-1] = \CDes(u^{-1}) \cap [n-1]$ and $v^j(w) = v^j(u)$ for $2 \leq j \leq n+1$.
For $j = 1$, since $w_2 \geq 2$,
\[
v^1(w) - w^1(u) = (v^{2}(w) + \tfrac{1}{2} e_{w_1}) - (v^{2}(u) + \tfrac{1}{2} e_{\ol{w_1}}) = e_{w_1} \in \tfrac{1}{2} \Phi^+_{C_n}.
\]

\noindent\textbf{Case 2:}
If $u = w_1 \dots w_{i+1} w_i \dots w_n$ for some $i \in [n-1]$ and $w_i \ll w_{i+1}$, then
\[
\CDes(w^{-1}) = \CDes(u^{-1}),
\]
since $w_i$ and $w_{i+1}$ (and $\ol{w_i}$ and $\ol{w_{i+1}}$) are not cyclically adjacent and they are the only entries that change positions between $\wt{u}$ and $\wt{w}$.
Thus, $v^j(w) = v^j(u)$ for $j \neq i+1$.
For $j = i+1$, since $w_i \ll w_{i+1}$,
\[
v^{i+1}(w) - v^{i+1}(u) = (v^{i+2}(w) + \tfrac{1}{2} e_{w_{i+1}}) - (v^{i+2}(u) + \tfrac{1}{2} e_{w_i}) = \tfrac{1}{2} (e_{w_{i+1}} - e_{w_{i}}) \in \tfrac{1}{2} \Phi^+_{C_n}.
\]

\noindent\textbf{Case 3:}
If $w_n = \ol{r}$ for some $r \in [2 , n]$ and $u = w_1 w_2 \dots w_{n-1} r$, then
\[
\CDes(w^{-1}) =
\begin{cases}
\CDes(u^{-1}) \setminus \{ n \} \sqcup \{n-1 , \ol{n} \} & \text{if } r = n,\\
\CDes(u^{-1}) \setminus \{ r,\ol{r+1} \} \sqcup \{ r-1 , \ol{r} \} & \text{if } r \leq n-1.
\end{cases}
\]
In both cases, $|\CDes(w^{-1}) \cap [j-1]| = |\CDes(u^{-1}) \cap [j-1]| + \delta_{j,r}$. Hence,
\[
v^{n+1}(w) - v^{n+1}(u) = e_r \in \tfrac{1}{2} \Phi^+_{C_n}.
\]
Moreover, since
\[
v^n(w) =
v^{n+1}(w) + \tfrac{1}{2} e_{\ol{r}} =
(v^{n+1}(u) + e_r) + \tfrac{1}{2} e_{\ol{r}} =
v^{n+1}(u) + \tfrac{1}{2} e_r =
v^n(u),
\]
it follows that $v^j(w) = v^j(u)$ for all $j \leq n$.
\end{proof}

\section{On the $h^*$-polynomial of the closed hypersimplices}

In this section, we relate the $h^*$-polynomials of the closed $\Delta_{C_n,k}$ and the half-open $\Delta'_{C_n,k}$ hypersimplices.

\begin{proposition}
For all $n > 0$ and $1 \leq k \leq 2n-1$,
\begin{align*}
h^*_{\Delta_{C_n,k}}(t)
= & h^*_{\Delta'_{C_n,k}}(t) + (1-t)\left( h^*_{\Delta'_{C_{n-1},k-1}}(t) + h^*_{\Delta_{C_{n-1},k-2}}(t) \right) \\
= & h^*_{\Delta'_{C_n,k}}(t) + \sum_{j \geq 1}(1-t)^j \left( h^*_{\Delta'_{C_{n-j},k-2j+1}}(t) + h^*_{\Delta'_{C_{n-j},k-2j}}(t) \right).
\end{align*}
\end{proposition}

\begin{proof}
We only prove the first equality since the second follows from it by induction.

Recall from \Cref{eq:def-half-open} that
\[
\Delta_{C_n,k} = \Delta'_{C_n,k} \sqcup \left( \Delta_{C_n,k} \cap H_{2e_n,k-1} \right).
\]
The intersection $\Delta_{C_n,k} \cap H_{2e_n,k-1} = \Pi_{C_n} \cap H_{2e_n,k-1}$ is determined by the following (in)equalities
\[
0 \leq 2x_1 , x_2 - x_1 , \dots , x_n - x_{n-1} \leq 1
\qquad\text{and}\qquad
2 x_n = k-1.
\]
Observe that $0 \leq \tfrac{k-1}{2} - x_{n-1} \leq 1$ if and only if $k-3 \leq 2 x_{n-1} \leq k-1$.
Moreover, the linear isomorphism $H_{2e_n,k-1} \to \RR^{n-1}$ given by the projection that forgets the last coordinate
induces an isomorphism on the lattices $\tfrac{1}{2} \ZZ^n \cap H_{2e_n,k-1}$ and $\ZZ^{n-1}$.
Therefore, $\Delta_{C_n,k} \cap H_{2e_n,k-1}$ is integrally equivalent to
$\Delta_{C_{n-1},k-2} \cup \Delta_{C_{n-1},k-1} = \Delta_{C_{n-1},k-2} \sqcup \Delta'_{C_{n-1},k-1}$.
The result follows since these polytopes have dimension one less that $\Delta_{C_n,k}$.
\end{proof}

\begin{example}
We apply the formula above for $n=3$ and $k = 2,k = 4$.
\begin{align*}
h^*_{\Delta_{C_3,2}}(t) = & h^*_{\Delta_{C'_3,2}}(t) + (1-t)h^*_{\Delta'_{C_2,1}}(t) \\
= & (5t+t^2) + (1-t) = 1 + 4t + t^2.
\end{align*}
\begin{align*}
h^*_{\Delta_{C_3,4}}(t) = & h^*_{\Delta_{C'_3,4}}(t) + (1-t)\left(h^*_{\Delta'_{C_2,3}}(t) + h^*_{\Delta'_{C_2,2}}(t) \right) + (1-t)^2 h^*_{\Delta'_{C_1,1}}(t) \\
= & (3t+3t^2) + (1-t)(t + 2t) + (1-t)^2 = 1 + 4t + t^2.
\end{align*}
More values of $h^*_{\Delta_{C_n,k}}(t)$ can be found in \cref{tab:h-closed}.
\end{example}

\section{On the $h^*$-polynomial of the fundamental parallelepiped}
\label{s:Psi}

We now turn our attention to the $h^*$-polynomial of the fundamental parallelepiped $\Pi_{C_n}$, which, by \cref{cor:h*-parallel}, corresponds to the polynomial $\Psi_{C_n}(t)$ defined in \Cref{eq:def-Psi}. For $n \geq 1$ and $0 \leq k \leq n$, we define
\[
\defnn{\psi_{n,k}} := | \{ w \in X_n \mid \basc(w) = k\}|,
\]
so that $\Psi_{C_n}(t) = \sum_k \psi_{n,k} t^k$. We will show that the numbers $\psi_{n,k}$ exhibit behavior analogous to the Eulerian numbers. To begin, we present an explicit recursion for $\psi_{n,k}$, which can be proven combinatorially.

\begin{proposition}
\label{p:recurrence}
The numbers $\psi_{n,k}$ are completely determined by the following recurrence.
For all $n > 1$,
\begin{equation}
\label{eq:recurrence}
\psi_{n,k} = (2n-2k+1) \psi_{n-1,k-1} + (2k+1) \psi_{n-1,k},
\end{equation}
with initial conditions $\psi_{1,0} = 1$ and $\psi_{1,k} = 0$ for $k \neq 0$.
\end{proposition}

\begin{proof}
Given $w \in X_n$, consider the signed permutation (in $X_{n-1}$) obtained by removing $\pm n$ from the window notation of $w$.
Then, the statistic $\basc$ is either unchanged or it decreases by exactly~$1$.
Indeed, to increase the statistic $\basc$ by removing a letter $w_i$, we need $w_{i-1} \not\!\!\ll w_i \not\!\!\ll w_{i+1}$ and $w_{i-1} \ll w_{i+1}$.
But this can only occur if $w_{i-1}, w_i, w_{i+1}$ are consecutive in the usual order of $\dbra{n}$.
Since we are removing either the minimum ($\ol{n}$) or maximum ($n$) element of $\dbra{n}$, $\basc$ cannot increase.

To conclude the recurrence in \eqref{eq:recurrence}, we show that if $w \in X_{n-1}$ has $\basc(w) = k$,
then exactly $2k + 1$ out of the $2n$ possible insertions of $\pm n$ do not increase the statistic $\basc$.

Let $w \in X_{n-1}$ with $\basc(w) = k$.
Take a big ascent position $i \in \BAsc(w)$, including the edge cases $i = \ol{1}$ ($w_1 \leq \ol{2}$) and $i = n-1$ ($w_{n-1} \geq 2$).
Note that $w_i \ll w_{i^+}$ implies both $w_i \neq n-1$ and $w_{i^+} \neq \ol{n-1}$,
the maximum and minimum elements, respectively.
Then, inserting $n$ ($w_i \ll n > w_{i+1}$)
or $\ol{n}$ ($w_i > \ol{n} \ll w_{i+1}$) after $w_i$,
does not change the statistic $\basc$.
In addition, if $n-1$ (resp. $\ol{n-1}$) appears in the window notation of in $w$,
we can insert $n$ after $n-1$ (resp. $\ol{n}$ before $\ol{n-1}$) without changing $\basc$.
Inserting $\pm n$ in any other position will increase $\basc$ by one
(either $w_{j} < n \gg w_{j+1}$ or $w_j \gg \ol{n} < w_{j+1}$).
\end{proof}

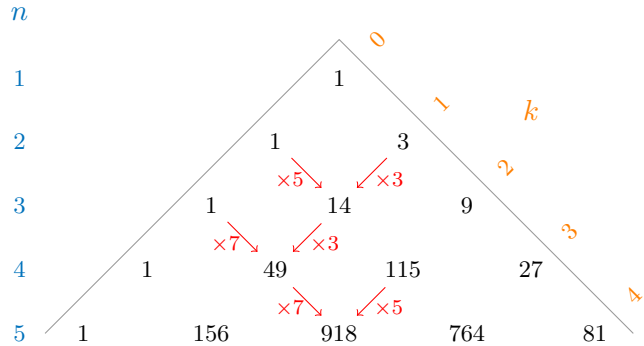
\begin{figure}[ht]
\centering
\begin{tikzpicture}
[scale = .85]
\node (10) at ( 0, 0) {\footnotesize $1$};
\node (20) at (-1,-1) {\footnotesize $1$};
\node (21) at ( 1,-1) {\footnotesize $3$};
\node (30) at (-2,-2) {\footnotesize $1$};
\node (31) at ( 0,-2) {\footnotesize $14$};
\node (32) at ( 2,-2) {\footnotesize $9$};
\node (40) at (-3,-3) {\footnotesize $1$};
\node (41) at (-1,-3) {\footnotesize $49$};
\node (42) at ( 1,-3) {\footnotesize $115$};
\node (43) at ( 3,-3) {\footnotesize $27$};
\node (50) at (-4,-4) {\footnotesize $1$};
\node (51) at (-2,-4) {\footnotesize $156$};
\node (52) at ( 0,-4) {\footnotesize $918$};
\node (53) at ( 2,-4) {\footnotesize $764$};
\node (54) at ( 4,-4) {\footnotesize $81$};
\node[] (n1) at (-5, 0) [color=RoyalBlue] {\footnotesize $1$};
\node[] (n2) at (-5,-1) [color=RoyalBlue] {\footnotesize $2$};
\node[] (n3) at (-5,-2) [color=RoyalBlue] {\footnotesize $3$};
\node[] (n4) at (-5,-3) [color=RoyalBlue] {\footnotesize $4$};
\node[] (n5) at (-5,-4) [color=RoyalBlue] {\footnotesize $5$};
\node (n) at (-5,1) [color=RoyalBlue] {$n$};
\node (k) at (3,-.5) [color=orange] {$k$};
\node[rotate=45] (k1) at ( .6, .6) [color=orange]{\footnotesize $0$};
\node[rotate=45] (k2) at (1.6,- .4) [color=orange]{\footnotesize $1$};
\node[rotate=45] (k3) at (2.6,-1.4) [color=orange]{\footnotesize $2$};
\node[rotate=45] (k4) at (3.6,-2.4) [color=orange]{\footnotesize $3$};
\node[rotate=45] (k5) at (4.6,-3.4) [color=orange]{\footnotesize $4$};
\draw[red,->] (21) -- (31) node[near end,right] {\scriptsize $\times3$};
\draw[red,->] (31) -- (41) node[near end,right] {\scriptsize $\times3$};
\draw[red,->] (42) -- (52) node[near end,right] {\scriptsize $\times 5$};;
\draw[red,->] (20) -- (31) node[near end,left] {\scriptsize $\times 5$};
\draw[red,->] (30) -- (41) node[near end,left] {\scriptsize $\times 7$};
\draw[red,->] (41) -- (52) node[near end,left] {\scriptsize $\times 7$};
\draw[black!40,shift={(0,.6)}]
(0,0) -- ( 4.6,-4.6)
(0,0) -- (-4.6,-4.6);
\end{tikzpicture}
\caption{The first values of $\psi_{n,k}$ computed using the recurrence in \Cref{eq:recurrence}.
For instance, $\psi_{5,2} = (2 \cdot 5 - 2 \cdot 2 + 1) \cdot 49 + (2 \cdot 2 + 1) \cdot 115 = 918$.
Each number is the weighted sum of the values directly north-east and north-west of it,
and the weights are constant along the diagonals.
See the last row of \cref{tab:h-half-open} for more values of $\psi_{n,k}$.
}
\label{f:TrianBasc}
\end{figure}

\begin{corollary}[Linear polynomial recurrence]\label{c:recurrence}
For any $n \geq 2$,
\[
\Psi_{C_n}(t) = (1 + (2n-1)t) \Psi_{C_{n-1}}(t) + 2t(1-t)\Psi'_{C_{n-1}}(t).
\]
\end{corollary}

\begin{remark}
The type B Eulerian polynomials $E_{B_n}(t) := \sum_k B_{n,k} t^k$ also satisfy the linear polynomial recurrence in \cref{c:recurrence};
see for instance Petersen's book \cite[Theorem 13.2]{petersen15}.
The two families of polynomials are therefore distinguished by their initial conditions:
$\Psi_{C_1}(t) = 1$ while $E_{B_1}(t) = 1 + t$.
We discuss further relations between $\Psi_{C_n}(t)$ and $E_{B_n}(t)$ in \cref{s:flag-des} below.
\end{remark}

One can use the recurrence in \cref{c:recurrence} to write a differential equation that determines the generating function of the polynomials $\Psi_{C_n}(t)$ and try to deduce a closed formula for it. Instead, we give a more elegant derivation of this generating function using Ehrhart theory.

\subsection{Generating function of big ascents}

Let $\defnn{{\rm Eul}_A(t,x)}$ be the exponential generating function for the classical Eulerian polynomials $E_{A_n}(t)$:
\begin{equation}
\label{eq:generatin-Eulerian}
\defnn{{\rm Eul}_A(t,x)} := \sum_{n \geq 0} E_{A_n}(t) \dfrac{x^n}{n!} = \dfrac{t-1}{t - e^{x(t-1)}}.
\end{equation}
We refer the reader to Foata's survey \cite[Section 3]{foata10eulerian} for a derivation of this formula.

\vspace{5pt}
\noindent\textbf{\cref{t:generating}.}
The following identity holds
\[
\sum_{n \geq 0} h^*_{\Pi_{C_{n+1}}}(t) \frac{x^n}{n!} = e^{3x(t-1)} {\rm Eul}_A(t,2x)^2.
\]

\begin{proof}
Using the same change of coordinates as in the proof of \cref{p:aux2}, we deduce that $\Pi_{C_n}$, with lattice $\tfrac{1}{2}\ZZ^n$, is integrally equivalent to the box $[0,1] \times [0,2]^{n-1}$, with the usual integer lattice $\ZZ^n$.
Hence, the Ehrhart polynomial of $\Pi_{C_n}$ is $L_{\Pi_{C_n}}(k) = (k+1) (2k+1)^{n-1}$.
Therefore,
\begin{equation}\label{eq:generating-aux}
\dfrac{\Psi_{C_n}(t)}{(1-t)^{n+1}}
= \sum_{k \geq 0} (2k+1)^{n-1} (k+1) t^k.
\end{equation}
Multiplying both sides of the equation above by $\tfrac{x^{n-1}}{(n-1)!}$ and taking the sum over $n \geq 1$, we obtain
\begin{multline*}
\dfrac{1}{(1-t)^2} \sum_{n \geq 0} \dfrac{\Psi_{C_{n+1}}(t)}{(1-t)^{n}} \dfrac{x^n}{n!} =
\sum_{n,k \geq 0} (2k+1)^n (k+1) t^k \dfrac{x^n}{n!} = \\
\sum_{k \geq 0} (k+1) t^k e^{((2k+1)x)} =
e^x \sum_{k \geq 0} (k+1) ( t e^{2x} )^k =
\dfrac{e^x}{(1-te^{2x})^2}.
\end{multline*}
Finally, multiplying both sides by $(t-1)^2$ and substituting $x := x(1-t)$ we get
\[
\sum_{n \geq 0} \Psi_{C_{n+1}}(t) \dfrac{x^n}{n!} =
\dfrac{e^{x(1-t)} (t-1)^2}{(1 - t e^{2x(1-t)})^2} \cdot \dfrac{e^{4x(t-1)}}{e^{4x(t-1)}} =
e^{3x(t-1)} \left( \dfrac{t-1}{t - e^{2x(t-1)}} \right)^2.
\]
The result follows by comparing the right hand side with \Cref{eq:generatin-Eulerian}.
\end{proof}

As a consequence of the proof, we show that the polynomials $\Psi_{C_n}(t)$ have real roots for all $n \geq 1$.
We use the following result of Brenti \cite[Theorem 2.2]{brenti94cox-eul}, which he used to show the real-rootedness of the Eulerian polynomials of type B.

\begin{theorem}[Brenti]
Let
\[
f(t) = \sum_{i=0}^d b_i \binom{t + d - i}{d}
\]
be a polynomial of degree $d$.
Suppose that $f(t)$ has all its roots in the interval $[-1,0]$.
Then the polynomial $\sum_{i=0}^d b_i t^i$ has only real roots.
\end{theorem}

\begin{theorem}\label{t:real-roots}
For all $n \geq 1$, the polynomial $\Psi_{C_n}(t)$ is real-rooted.
\end{theorem}

\begin{proof}
\Cref{eq:generating-aux} shows that
\[
\sum_k \psi_{n,k} \binom{t+n-k}{n} = (2t+1)^{n-1} (t+1).
\]
Clearly, all the roots of $(2t+1)^{n-1} (t+1)$, namely $-1$ and $-\tfrac{1}{2}$, lie on the interval $[-1,0]$.
Therefore, by Brenti's theorem above, the polynomial $\Psi_{C_n}(t) = \sum_k \psi_{n,k} t^k$ is real-rooted.
\end{proof}

\section{Relations with type B Eulerian numbers, two ways}
\label{s:flag-des}

In type A, the Eulerian numbers appear both as the $h^*$-polynomial of the fundamental parallelepiped (or equivalently, as the cover-counting statistic on the poset of cyclic permutations of \cite{ACR21}) and as the volumes of the hypersimplices. In this section, we explore analogous results for type C. An analog of the volume formula was not expected, as there are $2n-1$ hypersimplices of type $C_n$ but only $n+1$ possible values for an Eulerian statistic on $\hypoct_n$.

\subsection{Big ascents and Coxeter descents}

If we use \cref{t:fexc-desB} instead of \cref{t:cdes-basc} to compute the $h^*$-polynomial of $\Pi_{C_n}$ from that of the half-open hypersimplices $\Delta'_{C_n,k}$, we obtain the following result.

\begin{corollary}
\label{t:half-weak1}
For all $n \geq 1$,
\[
\Psi_{C_n}(t) = h^*_{\Pi_{C_n}}(t) = \displaystyle\sum_{ w \in X_n } t^{\des_B(w)}.
\]
\end{corollary}

Since for all $w \in \hypoct_n$ we have $\des_B(w) + \des_B(\ol{w}) = n$, and $\hypoct_n = X_n \sqcup \ol{X_n}$, we deduce the following.

\begin{corollary}\label{t:half-weak2}
For all $n \geq 1$,
\[
\Psi_{C_n}(t) + t^n \Psi_{C_n}(t^{-1}) = E_{B_n}(t),
\]
where $E_{B_n}(t)$ is the type B Eulerian polynomial.
\end{corollary}

For example, for $n = 3$ we have
\[
\Psi_{C_3}(t) + t^3 \Psi_{C_3}(t^{-1}) = (1 + 14 t + 9 t^2) + (9 t + 14 t^2 + t^3) = 1 + 23 t + 23 t^2 + t^3 = E_{B_3}(t).
\]

\begin{remark}
\cref{t:half-weak1} implies that the polynomial $\Psi_{C_n}(t)$ agrees with the polynomial defined by Matthew Hyatt in \cite[Theorem 1.1]{Hyatt16Recurrences}. Using different methods, Hyatt also proves the real-rootedness of these polynomials (c.f. \cref{t:real-roots}).
\end{remark}

\subsection{Volumes of type C hypersimplices and type B Eulerian numbers}

Let $w_0 = \ol{n} \dots \ol{2} \, \ol{1}$ be the longest element in $\hypoct_n$. Then, $\cdes(w_0 w) = \cdes(w)$ for any signed permutation $w$. Thus, plugging $t = 1$ in \cref{t:cdes-basc} recovers the formula for the volume of the type C hypersimplices from Lam-Postnikov; \Cref{eq:LP-volume}:
\[
\Vol(\Delta_{C_n,k}) = |\{ w \in X_n \mid \cdes(w) = k \}| = \dfrac{1}{2} |\{ w \in \hypoct_n \mid \cdes(w) = k \}|.
\]

Plugging $t = 1$ in \cref{t:fexc-desB}, yields a different formula.

\vspace{5pt}
\noindent\textbf{\cref{cor:volume-fexc}.}
For all $n \geq 1$ and $k \geq 1$,
\[
\Vol(\Delta_{C_n,k}) = |\{ w \in X_n \mid \fexc(w) = k-1 \}|.
\]

By comparing both formulas, one deduces that $\fexc + 1$ and $\cdes$ have the same distribution on~$X_n$. However, their distribution over all $\hypoct_n$ differ; as $\fexc$ takes $2n$ different values while $\cdes$ only takes $2n-1$ values.

\vspace{5pt}
\noindent\textbf{\cref{cor:volC-EulB}.}
For all $n \geq 1$ and $k \geq 1$, we have:
\[
B_{n,k} = \Vol(\Delta_{C_n,2k-1}) + 2 \Vol(\Delta_{C_n,2k}) + \Vol(\Delta_{C_n,2k+1}),
\]
where $\Vol(\Delta_{C_n,j}) = 0$ whenever $j < 1$ or $j > 2n-1$.

\begin{proof}
Recall from \cref{ss:hypoct} that the statistic $\exc_B(w) := \left\lfloor \tfrac{\fexc(w) + 1}{2} \right\rfloor$ is Eulerian. Moreover, a signed permutation $w \in X_n$ satisfies $\fexc(w) = r$ if and only if the signed permutation $w' \in \ol{X_n}$ obtained from $w$ by negating $1$ satisfies $\fexc(w') = r+1$. Thus,
\begin{align*}
B_{n,k} &= |\{ w \in \hypoct_n \mid \exc_B(w) = k \}| \\
&= |\{ w \in \hypoct_n \mid \fexc(w) = 2k-1 \}| + |\{ w \in \hypoct_n \mid \fexc(w) = 2k \}| \\
&= |\{ w \in X_n \mid \fexc(w) = 2k-2 \}| + 2 |\{ w \in X_n \mid \fexc(w) = 2k-1 \}| + |\{ w \in X_n \mid \fexc(w) = 2k \}| \\
&= \Vol(\Delta_{C_n,2k-1}) + 2 \Vol(\Delta_{C_n,2k}) + \Vol(\Delta_{C_n,2k+1}).
\end{align*}
The last step uses \cref{cor:volume-fexc}.
\end{proof}

\section{Limit Poset: Strict Partitions}\label{s:limit}

\cref{t:generating} provides a simple description of the exponential generating function for the distribution of cover relations in $\mathcal{I}_{\Phi_{C_n}} \cong (X_n,\leq)$. This prompts natural questions about the asymptotic behavior of these partial orders. In type A, Laget-Chapelier, Reutenauer, and the first author \cite{ACR21} proved that $\mathcal{I}_{\Phi_{A_n}}$, the weak order on the alcoves contained in $\Pi_{A_n}$, converges to Young's lattice of partitions as $n$ goes to infinity. We establish an analogous result for $(X_n,\leq)$.

Let $\defnn{\mathcal{P}}$ denote \defn{Young's lattice}, which is the partial order on integer partitions given by the containment of their corresponding Ferrers diagrams. The collection of \defn{strict partitions}--partitions in which all parts are distinct--forms the sublattice $\defnn{\mathcal{SP}}$. \cref{f:limiting_poset} (right) shows the Hasse diagram of $\mathcal{SP}$ truncated at rank $6$.

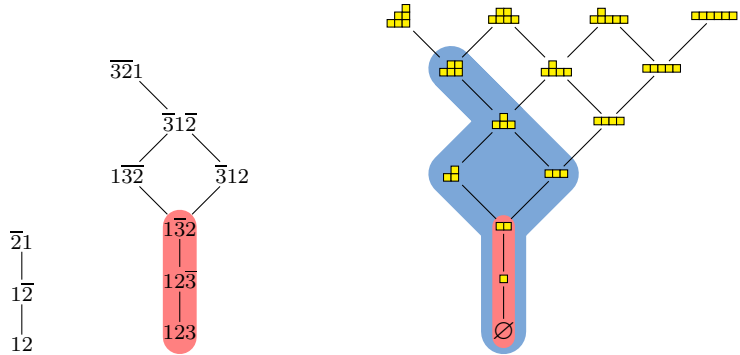
\begin{figure}[ht]
\centering
\begin{tikzpicture}[scale =.7]
\node[inner sep = 1pt] (0000000) at ( 0, 0) {\scriptsize $12$};
\node[inner sep = 1pt] (1000000) at ( 0, 1) {\scriptsize $1\ol{2}$};
\node[inner sep = 1pt] (0100000) at ( 0, 2) {\scriptsize $\ol{2}1$};
\draw[] (0000000) -- (1000000);
\draw[] (1000000) -- (0100000);
\end{tikzpicture} \hspace{.7cm}
\begin{tikzpicture}[scale =.7]
\node[inner sep = 1pt] (0000000) at ( 0, 0) {\scriptsize $123$};
\node[inner sep = 1pt] (1000000) at ( 0, 1) {\scriptsize $12\ol{3}$};
\node[inner sep = 1pt] (0100000) at ( 0, 2) {\scriptsize $1\ol{3}2$};
\node[inner sep = 1pt] (1100000) at (-1, 3) {\scriptsize $1\ol{3}\ol{2}$};
\node[inner sep = 1pt] (0010000) at ( 1, 3) {\scriptsize $\ol{3}12$};
\node[inner sep = 1pt] (1010000) at ( 0, 4) {\scriptsize $\ol{3}1\ol{2}$};
\node[inner sep = 1pt] (0110000) at (-1, 5) {\scriptsize $\ol{3}\ol{2}1$};
\draw[] (0000000) -- (1000000);
\draw[] (1000000) -- (0100000);
\draw[] (0100000) -- (1100000);
\draw[] (0100000) -- (0010000);
\draw[] (1100000) -- (1010000);
\draw[] (0010000) -- (1010000);
\draw[] (1010000) -- (0110000);
\begin{scope}[on background layer]
\path[fill=red!50,draw=red!50,line width=.45cm, line cap=round, line join=round]
(0,-.1) -- (0,2);
\end{scope}
\end{tikzpicture} \hspace{1.5cm}
\begin{tikzpicture}[scale =.7]
\node[inner sep = 1pt] (0000000) at ( 0, 0) {$\varnothing$};
\node[inner sep = 1pt] (1000000) at ( 0, 1) {\YS\YB{0}{0}\YF};
\node[inner sep = 1pt] (0100000) at ( 0, 2) {\YS\YB{0}{0}\YB{1}{0}\YF};
\node[inner sep = 1pt] (1100000) at (-1, 3) {\YS\YB{0}{0}\YB{1}{0}\YB{1}{1}\YF};
\node[inner sep = 1pt] (0010000) at ( 1, 3) {\YS\YB{0}{0}\YB{1}{0}\YB{2}{0}\YF};
\node[inner sep = 1pt] (1010000) at ( 0, 4) {\YS\YB{0}{0}\YB{1}{0}\YB{2}{0}\YB{1}{1}\YF};
\node[inner sep = 1pt] (0001000) at ( 2, 4) {\YS\YB{0}{0}\YB{1}{0}\YB{2}{0}\YB{3}{0}\YF};
\node[inner sep = 1pt] (0110000) at (-1, 5) {\YS\YB{0}{0}\YB{1}{0}\YB{2}{0}\YB{1}{1}\YB{2}{1}\YF};
\node[inner sep = 1pt] (1001000) at ( 1, 5) {\YS\YB{0}{0}\YB{1}{0}\YB{2}{0}\YB{3}{0}\YB{1}{1}\YF};
\node[inner sep = 1pt] (0000100) at ( 3, 5) {\YS\YB{0}{0}\YB{1}{0}\YB{2}{0}\YB{3}{0}\YB{4}{0}\YF};
\node[inner sep = 1pt] (1110000) at (-2, 6) {\YS\YB{0}{0}\YB{1}{0}\YB{2}{0}\YB{1}{1}\YB{2}{1}\YB{2}{2}\YF};
\node[inner sep = 1pt] (0101000) at ( 0, 6) {\YS\YB{0}{0}\YB{1}{0}\YB{2}{0}\YB{3}{0}\YB{1}{1}\YB{2}{1}\YF};
\node[inner sep = 1pt] (1000100) at ( 2, 6) {\YS\YB{0}{0}\YB{1}{0}\YB{2}{0}\YB{3}{0}\YB{4}{0}\YB{1}{1}\YF};
\node[inner sep = 1pt] (0000010) at ( 4, 6) {\YS\YB{0}{0}\YB{1}{0}\YB{2}{0}\YB{3}{0}\YB{4}{0}\YB{5}{0}\YF};
\draw[] (0000000) -- (1000000);
\draw[] (1000000) -- (0100000);
\draw[] (0100000) -- (1100000);
\draw[] (0100000) -- (0010000);
\draw[] (1100000) -- (1010000);
\draw[] (0010000) -- (1010000);
\draw[] (0010000) -- (0001000);
\draw[] (1010000) -- (0110000);
\draw[] (1010000) -- (1001000);
\draw[] (0001000) -- (1001000);
\draw[] (0001000) -- (0000100);
\draw[] (0110000) -- (1110000);
\draw[] (0110000) -- (0101000);
\draw[] (1001000) -- (0101000);
\draw[] (1001000) -- (1000100);
\draw[] (0000100) -- (1000100);
\draw[] (0000100) -- (0000010);
\begin{scope}[on background layer]
\path[fill=RoyalBlue!50,draw=RoyalBlue!50,line width=.6cm, line cap=round, line join=round]
(0,0) -- (0,2) -- (-1,3) -- (0,4) -- (-1,5) -- (1,3) -- (0,2);
\path[fill=red!50,draw=red!50,line width=.3cm, line cap=round, line join=round]
(0,-.1) -- (0,2);
\end{scope}
\end{tikzpicture}
\caption{
Subposets $(Y_n , \leq)$ for $n = 2,3$ (left),
and their embedding in the lattice of strict partitions (right),
represented using \emph{shifted diagrams} (in French notation).}
\label{f:limiting_poset}
\end{figure}

Consider the order ideal $\defnn{Y_n} := \{w \in X_n \mid \cdes(w^{-1}) \leq 2\}$ of $X_n$.
It consists of those signed permutations $w \in X_n$ whose associated alcoves $A(w)$ lie inside $\Delta_{C_n,1} \cup \Delta_{C_n,2}$.
Let $\tau_n : Y_n \to X_{n+1}$ be the function $\tau(w_1 w_2 \dots w_n) = 1 w'_1\, w'_2 \cdots w'_n$ where $w_i' \in \dbra{n+1}$ has the same sign as $w_i$ and satisfies $|w'_i| = |w_i|+1$.
For example,
\[
\tau_2(1\ol{2}) = 1 2\ol{3}
\qquad\text{and}\qquad
\tau_3(\ol{3} 1 \ol{2}) = 1 \ol{4} 2 \ol{3}.
\]

\begin{theorem}
For all $n \geq 1$, $\tau_n(Y_n) \subset Y_{n+1}$ and $\tau_n : Y_n \to Y_{n+1}$ is a poset embedding.
Furthermore, the lattice of strict partitions $\mathcal{SP}$ is the colimit in the category of posets of the diagram
\[
Y_1 \xrightarrow{\tau_1} Y_2 \xrightarrow{\tau_2} Y_3 \xrightarrow{\tau_3} \dots .
\]
\end{theorem}

\begin{proof}
The hyperplane separating $\Delta_{C_n,2}$ and $\Delta_{C_n,3}$ is given by the equation $\langle x , 2 e_n \rangle = 2$.
Thus, the elements $w \in Y_n$ are those obtained from the identity $1 \mkern 1mu 2 \dots n$ by applying the relations in \cref{eq:cover-rel} and changing the sign of $n$ at most once.
Therefore, the window notation of elements $w \in Y_n$ are precisely the shuffles of the words $1 \mkern 1mu 2 \dots k$ and $\ol{n} \dots \ol{k+1}$ for some $1 \leq k \leq n$.
Define a map $\lambda_n : Y_n \to \mathcal{SP}$ by
\[
\defnn{\lambda_n(w)} := ( n + 1 - i \mid i \in [n] \,,\, w_i < 0 ).
\]
For example,
$\lambda_6(1 \mkern 1mu 2 \mkern 1mu 3 \mkern 1mu 4 \mkern 1mu 5 \mkern 1mu \ol{6}) = (1) = {\begin{gathered}\begin{tikzpicture}[scale = .2]\YB{0}{0}\YF\end{gathered}}$, and
$\lambda_6(1 \mkern 1mu \ol{6} \mkern 1mu 2 \mkern 1mu 3 \mkern 1mu \ol{5} \mkern 1mu \ol{4}) = (5,2,1) = {\begin{gathered}\begin{tikzpicture}[scale = .2]\YB{0}{0}\YB{1}{0}\YB{2}{0}\YB{3}{0}\YB{4}{0}\YB{1}{1}\YB{2}{1}\YB{2}{2}\YF\end{gathered}}$.
It follows form the form of the cover relations in $Y_n$ that $\lambda_n$ sends cover relations in $Y_n$ to cover relations in $\mathcal{SP}$.

Since $\tau_n$ only adds a positive letter in the first position, $\lambda_{n+1} \circ \tau_n = \lambda_n$.
This shows that $\mathcal{SP}$ (together with the maps $\lambda_n$) is a \emph{cocone} over the diagram $\big( Y_1 \xrightarrow{\tau_1} Y_2 \xrightarrow{\tau_2} \dots \big)$.
To prove that $\mathcal{SP}$ is in fact the colimit of this diagram, it suffices to observe that any strict partition $\mu$ with $\mu_1 \leq n$ determines a unique shuffle $w$ of $1,\dots,k$ and $\ol{n} \dots \ol{k+1}$ such that $\lambda_n(w) = \mu$ and $k = n - \ell(\mu)$: the shuffle for which the negative letters appear in positions
$n + 1 - \mu_1, n + 1 - \mu_2, \dots , n + 1 - \mu_{\ell(\mu)}$.
\end{proof}

Observe that the first step of the proof shows that, in particular, $Y_n$ contains all the elements $w \in X_n$ of rank at most $n$. Given that the morphisms $Y_{n-1} \to Y_{n}$ preserve rank, we deduce the following result.

\begin{corollary}
For every $N$ and $n \geq N$, the truncation of $Y_n$ to elements of rank at most $N$ is isomorphic to $\mathcal{SP}_{\leq N}$.
\end{corollary}

\begin{corollary}
For any $w \in Y_n$, the number of maximal chains ${\rm id}_n \rightharpoonup \dots \rightharpoonup w$ in $X_n$
is the number of shifted standard Young tableaux of shape $\lambda_n(w)$.
\end{corollary}

\begin{example}
The signed permutation $w = 1 \mkern 1mu 2 \mkern 1mu 3 \mkern 1mu \ol{6} \mkern 1mu 4 \mkern 1mu \ol{5} \in Y_6$ has negative entries in the fourth and sixth positions,
so $\lambda_6(w) = (7-4,7-6) = (3,1)$.
Hence, there are exactly two maximal chains in the interval $[{\rm id}_6, w]$ of $X_6$ (corresponding to the shifted SYT
${\begin{gathered}\YSF\YBF{0}{0}{1}\YBF{1}{0}{2}\YBF{2}{0}{3}\YBF{1}{1}{4}\YF\end{gathered}}$ and
${\begin{gathered}\YSF\YBF{0}{0}{1}\YBF{1}{0}{2}\YBF{2}{0}{4}\YBF{1}{1}{3}\YF\end{gathered}}$).
\end{example}

\section{Types B and D}\label{s:BandD}

We now have a solid understanding for the Ehrhart $h^*$-polynomials of hypersimplices for types A and C. In this section, we will explore what is fundamentally different for types B and D and conclude with an intriguing conjecture.

We begin with the case of type B. For $n = 1,2$, the root systems of types $B_n$ and $C_n$ are isomorphic. Thus, we assume $n \geq 3$. Let $\Phi \subset \RR^n$ be the root system of type $B_n$, whose set of simple roots given by
\[
\alpha_i = e_i - e_{i+1}, \hspace{.6cm}\text{for } 1 \leq i \leq n-1,
\hspace{1cm}\text{and}\hspace{1cm}
\alpha_n = e_n.
\]
Its highest root is
\[
\alpha_1 + 2 \alpha_2 + \dots + 2 \alpha_n = e_1 + e_2,
\]
and the nonzero vertices of the fundamental alcove are
\[
\omega_1 = e_1 ,\quad
\tfrac{\omega_2}{2} = \tfrac{1}{2} ( e_1 + e_2 ) ,\quad
\dots , \quad
\tfrac{\omega_{n-1}}{2} = \tfrac{1}{2} ( e_1 + \dots + e_{n-1}), \quad
\tfrac{\omega_n}{2} = \tfrac{1}{2} ( e_1 + \dots + e_n).
\]
The lattice $\mathcal L$ spanned by the vertices of the fundamental alcove is
\[
\mathcal{L} = \{ v \in \tfrac{1}{2}\ZZ^n \mid v_1 - v_2 \in \ZZ \}.
\]
Unlike in type C, the lattice spanned by the vertices of the fundamental alcove does not agree with the set of vertices of all the alcoves in $\HH_{\Phi}$. That is to say, $\mathcal{L}$ is not invariant under the action of $\wt{W}_\Phi$. For instance, the reflection of the point $\tfrac{\omega_2}{2} = (\frac{1}{2},\frac{1}{2},0) \in \mathcal{L}$ across the hyperplane $H_{e_2-e_3,0} \in \HH$ is the point $(\frac{1}{2},0,\frac{1}{2}) \notin \mathcal{L}$. Therefore, not all alcoves are lattice polytopes with respect to $\mathcal{L}$. Alternatively, if we consider the lattice spanned by the vertices of all alcoves, namely $\tfrac{1}{2}\ZZ^n$, we find that the alcoves are not unimodular; in fact, they have a normalized volume of $2$.

The same is true in type $D_n$, for $ n \geq 4$. Namely, the lattice spanned by the vertices of the fundamental alcove is not invariant under the action of  $\wt{W}_{\Phi}$. In particular, we cannot apply the ideas in \cref{ss:proof2} to compute the $h^*$-polynomials of type B and D hypersimplices by studying the weak order restricted to the alcoves contained in the fundamental parallelepiped. However, computational evidence suggests an intriguing relationship between $\Psi_{B_n}(t)$ and the type D Eulerian polynomial $E_{D_n}(t)$.

\begin{conjecture}
For $n \geq 3$, 
\[{\Psi_{B_n}(t) + t^n \Psi_{B_n}(t^{-1})} = 2 E_{D_n}(t).\]
\end{conjecture}

For example, for $n = 4$,
\[
{(1 + 56 t + 102 t^2 + 32 t^3 + t^4) + (1 + 32 t + 102 t^2 + 56 t^3 + t^4)} = 2 (1 + 44 t + 102 t^2 + 44 t^3 + t^4).
\]
Compare the conjecture with \cref{t:half-weak2}.
\cref{tab:typesBD} below show the coefficients of the polynomials $\Psi_{B_n}(t)$ and $\Psi_{D_n}(t)$ for small values of $n$.

\begin{table}[ht]
\centering
\small
\begin{tabular}{c|l|l}
$n$ & $\Psi_{B_n}(t)$ & $\Psi_{D_n}(t)$ \\
\hline
$3$ & $1+15t+7t^2+t^3$ & $1+4t+t^2$ \\
$4$ & $1+56t+102t^2+32t^3+t^4$ & $1+22t+18t^2+6t^3+t^4$ \\
$5$ & $1+189t+898t^2+706t^3+125t^4+t^5$ & $1+85t+222t^2+138t^3+33t^4+t^5$ \\
$6$ & $1+610t+6351t^2+10876t^3+4751t^4+450t^5+t^6$ & $1+294t+1895t^2+2380t^3+1047t^4+142t^5+t^6$
\end{tabular}
\caption{The first values of $\Psi_{B_n}(t)$ and $\Psi_{D_n}(t)$.}
\label{tab:typesBD}
\end{table}

$\Psi_{D_n}(t)$ is not real-rooted for $n = 4,5,6$.
$\Psi_{B_n}$ is real-rooted for $n=4,5,6$ and not for $n=3$

\bibliographystyle{plain}
{
\bibliography{Bibliography}}

\end{document}